\documentclass[12pt,a4paper,leqno]{amsart} 
\usepackage[T1]{fontenc}       
\usepackage{amsfonts}          
\usepackage{amsmath}
\usepackage{amssymb}
\usepackage{overpic}
\usepackage{euscript} 
\usepackage{eurosym}
\usepackage{comment}
\usepackage{mathrsfs} 
\usepackage{ifthen}
\usepackage{esint} 
\usepackage{color}

\long\def\symbolfootnote[#1]#2{\begingroup%
\def\thefootnote{\fnsymbol{footnote}}\footnote[#1]{#2}\endgroup}

{\qed\vspace{5pt}}

\usepackage{amsthm}

\newtheoremstyle{lause}
{5pt}
{5pt}
{\slshape}
{\parindent}
{\bfseries}
{.}
{.5em}
{}

\theoremstyle{lause}

\newtheoremstyle{maaritelma}
{5pt}
{5pt}
{\rmfamily}
{\parindent}
{\bfseries}
{.}
{.5em}
{}

\theoremstyle{maaritelma}
\newtheoremstyle{lause}
{5pt}
{5pt}
{\slshape}
{\parindent}
{\bfseries}
{.}
{.5em}
{}

\theoremstyle{lause}
\newtheorem{theorem}{Theorem}[section]
\newtheorem{lemma}[theorem]{Lemma}
\newtheorem{proposition}[theorem]{Proposition}
\newtheorem{corollary}[theorem]{Corollary}

\newtheoremstyle{maaritelma}
{5pt}
{5pt}
{\rmfamily}
{\parindent}
{\bfseries}
{.}
{.5em}
{}

\theoremstyle{maaritelma}
\newtheorem{definition}[theorem]{Definition}

\newtheorem{example}[theorem]{Example}
\newtheorem{remark}[theorem]{Remark}

\usepackage{graphicx}
\usepackage{caption}
\DeclareMathOperator*{\essinf}{ess\,inf}

\numberwithin{equation}{section}

\begin{document}

\thispagestyle{empty}

\begin{center}

{\large{\textbf{Harmonic measure, equilibrium measure, and thinness at infinity in the theory of Riesz potentials}}}

\vspace{18pt}

\textbf{Natalia Zorii}

\vspace{18pt}

\emph{Dedicated to Professor Bent Fuglede on the occasion of his 95th birthday}\vspace{8pt}

\footnotesize{\address{Institute of Mathematics, Academy of Sciences
of Ukraine, Tereshchenkivska~3, 01601,
Kyiv-4, Ukraine\\
natalia.zorii@gmail.com }}

\end{center}

\vspace{12pt}

{\footnotesize{\textbf{Abstract.} Focusing first on the inner $\alpha$-har\-mo\-nic measure $\varepsilon_y^A$ ($\varepsilon_y$ being the unit Dirac measure, and $\mu^A$ the inner $\alpha$-Riesz balayage of a Radon measure $\mu$ to $A\subset\mathbb R^n$ arbitrary), we
describe its Euclidean support, provide a formula for evaluation of its total mass, establish the vague continuity of the map $y\mapsto\varepsilon_y^A$
outside the inner $\alpha$-ir\-reg\-ul\-ar points for $A$, and obtain necessary and sufficient conditions for $\varepsilon_y^A$ to be of finite energy (more generally, for $\varepsilon_y^A$ to be absolutely continuous with respect to inner capacity) as well as for $\varepsilon_y^A(\mathbb R^n)\equiv1$ to hold. Those criteria are given in terms of the newly defined concepts of $\alpha$-thin\-ness and $\alpha$-ul\-tra\-thin\-ness at infinity that generalize the concepts of thinness at infinity by Doob and Brelot, respectively. Further, we extend some of these results to $\mu^A$ general by verifying the
formula $\mu^A=\int\varepsilon_y^A\,d\mu(y)$. We also show that there is a $K_\sigma$-set $A_0\subset A$ such that $\mu^A=\mu^{A_0}$ for all $\mu$, and give various applications of this theorem. In particular, we prove the vague and strong continuity of the inner swept, resp.\ equilibrium, measure under the approximation of $A$ arbitrary, thereby    strengthening Fuglede's result established for $A$ Borel (Acta Math., 1960).
Being new even for $\alpha=2$, the results obtained also present a further development of the theory of inner Newtonian capacities and of inner Newtonian balayage, originated by Cartan.}}
\symbolfootnote[0]{\quad 2010 Mathematics Subject Classification:
Primary 31C15.} \symbolfootnote[0]{\quad Key words: inner Riesz balayage, inner $\alpha$-harmonic measure, inner Riesz equilibrium measure, inner $\alpha$-thin\-ness and $\alpha$-ul\-tra\-thin\-ness at infinity.}

\vspace{6pt}

\markboth{\emph{Natalia Zorii}} {\emph{$\alpha$-harmonic measure, $\alpha$-Riesz equilibrium measure, and $\alpha$-thinness at infinity}}

\section{Introduction and preliminaries}

This section provides a brief exposition of the theory of inner Riesz balayage of Radon measures on $\mathbb R^n$,
which has been developed in \cite{Z-bal} in the frame of the classical approach initiated for the Newtonian kernel
by Cartan \cite{Ca2}. We also recall the notion of inner Riesz equilibrium measure $\gamma_A$ for $A\subset\mathbb R^n$ arbitrary, treated in the extended sense where $\gamma_A(\mathbb R^n)$ might be infinite,
and give a summary of the results of the present study.
To begin with, we review some basic facts of potential theory with respect to the Riesz kernel $\kappa_\alpha(x,y):=|x-y|^{\alpha-n}$ of order $\alpha\in(0,2]$ on $\mathbb R^n$, $n\geqslant3$ (see \cite{L}).

\subsection{Some basic facts of Riesz potential theory}\label{sec-pr1} Let $\mathfrak M=\mathfrak M(\mathbb R^n)$ denote the linear space of all real-valued Radon measures $\mu$ on $\mathbb R^n$ equipped with the \emph{vague\/} topology of pointwise convergence on the class $C_0=C_0(\mathbb R^n)$ of all fi\-ni\-te-val\-ued continuous functions on $\mathbb R^n$ with compact support. Given $\mu,\nu\in\mathfrak M$, we define the \emph{potential\/} and the \emph{mutual energy\/} by
\begin{align*}\kappa_\alpha\mu(x)&:=\int\kappa_\alpha(x,y)\,d\mu(y),\quad x\in\mathbb R^n,\\
\kappa_\alpha(\mu,\nu)&:=\int\kappa_\alpha(x,y)\,d(\mu\otimes\nu)(x,y),
\end{align*}
respectively, provided the right-hand side is well defined as a finite number or $\pm\infty$. For $\mu=\nu$, $\kappa_\alpha(\mu,\nu)$ defines the \emph{energy\/} $\kappa_\alpha(\mu,\mu)$ of $\mu$. All $\mu\in\mathfrak M$ with $\kappa_\alpha(\mu,\mu)$ finite form a pre-Hil\-bert space $\mathcal E_\alpha=\mathcal E_\alpha(\mathbb R^n)$ with the inner product $\kappa_\alpha(\mu,\nu)$ and the norm $\|\mu\|_\alpha:=\sqrt{\kappa_\alpha(\mu,\mu)}$. The topology on $\mathcal E_\alpha$ defined by $\|\cdot\|_\alpha$ is said to be \emph{strong}.

For $A\subset\mathbb R^n$ arbitrary, we denote by $\mathfrak M^+(A)$ the cone of all positive $\mu\in\mathfrak M$ \emph{concentrated\/} on $A$, which means that $A^c:=\mathbb R^n\setminus A$ is  $\mu$-neg\-lig\-ible, or equivalently that $A$ is $\mu$-meas\-ur\-able and $\mu=\mu|_A$, where $\mu|_A$ is the restriction of $\mu$ to $A$. Write $\mathcal E^+_\alpha(A):=\mathcal E_\alpha\cap\mathfrak M^+(A)$, $\mathfrak M^+:=\mathfrak M^+(\mathbb R^n)$, and $\mathcal E^+_\alpha:=\mathcal E^+_\alpha(\mathbb R^n)$.

By Deny \cite{D1} (for $\alpha=2$, see also Cartan \cite{Ca}), the cone $\mathcal E^+_\alpha$ is strongly complete, and the strong topology on $\mathcal E^+_\alpha$ is stronger than the (induced) vague topology. This implies that for any $A\subset\mathbb R^n$ with finite \emph{inner capacity\/} $c_\alpha(A)$,
there exists the (\emph{inner\/}) \emph{equilibrium measure\/} $\gamma_A\in\mathcal E^+_\alpha$ which is uniquely determined by the two relations
\begin{align}\label{eq11}&\gamma_A(\mathbb R^n)=\|\gamma_A\|^2_\alpha=c_\alpha(A),\\
\label{eq22}&\kappa_\alpha\gamma_A(x)=1\text{ \ n.e.\ on\ }A.\end{align}
Here
\[1/c_\alpha(A):=\inf_{\mu\in\mathcal E^+_\alpha(A):\ \mu(\mathbb R^n)=1}\,\|\mu\|^2_\alpha,\]
and the abbreviation "$\mathcal U(x)$ \emph{n.e.\ on\/} $A$" means that the set of all $x\in A$ where the assertion $\mathcal U(x)$ fails has $c_\alpha(\cdot)=0$. This $\gamma_A$ can also be found as the unique solution to either of the extremal problems
\begin{align*}&\kappa_\alpha\gamma_A=\inf_{\nu\in\Theta_A}\,\kappa_\alpha\nu,\\
&\|\gamma_A\|_\alpha=\inf_{\nu\in\Gamma_A}\,\|\nu\|_\alpha,\end{align*}
where $\Theta_A$ consists of all $\nu\in\mathfrak M^+$ with $\kappa_\alpha\nu\geqslant1$ n.e.\ on $A$, and $\Gamma_A:=\Theta_A\cap\mathcal E^+_\alpha$.\footnote{This classical concept of inner Riesz equilibrium measure has been extended in \cite{Z-bal} to the case where $c_\alpha(A)$ might be infinite. See Section~\ref{sec-ext} below for some details of this generalization, and also Sections~\ref{sec-intim}, \ref{sec-desc}, \ref{sec-last-1}, and \ref{sec-last} for further relevant results.}

In what follows, when speaking of a measure $\mu\in\mathfrak M^+$, we always tacitly assume that its potential $\kappa_\alpha\mu$ is not identically infinite,
which according to \cite[Chapter~I, Section~3, n$^\circ$~7]{L} holds if and only if
\[\int_{|y|>1}\frac{d\mu(y)}{|y|^{n-\alpha}}<\infty.\]
By \cite[Chapter~III, Section~1, n$^\circ$~1]{L}, $\kappa_\alpha\mu$ is then finite, in fact, n.e.\ on $\mathbb R^n$.\footnote{Hence, $\kappa_\alpha\mu<\infty$ q.e.\ on $\mathbb R^n$, where "q.e." refers to an exceptional set of zero \emph{outer\/} capacity.}

A measure $\mu\in\mathfrak M^+$ is said to be \emph{$c_\alpha$-absolutely continuous\/} if $\mu(K)=0$ for every compact set $K\subset\mathbb R^n$ with $c_\alpha(K)=0$. This certainly occurs if $\kappa_\alpha(\mu,\mu)$ is finite or, more generally, if $\kappa_\alpha\mu$ is locally bounded (but not conversely, see \cite[pp.~134--135]{L}).

For $y\in\mathbb R^n$, define the inversion $J_y$ with respect to $S(y,1):=\{x:\ |x-y|=1\}$ mapping each point $x\ne y$ to the point $x^*=J_y(x)$ on the ray through $x$ issuing from $y$ which is uniquely determined by
\[|x-y|\cdot|x^*-y|=1.\]
This is a homeomorphism of $\mathbb R^n\setminus\{y\}$ onto itself having the property
\begin{equation}\label{inv}|x^*-z^*|=\frac{|x-z|}{|x-y||z-y|}\text{ \ for all\ }x,z\in\mathbb R^n\setminus\{y\},\end{equation}
and it can be extended to a homeomorphism of the one-point compactification $\overline{\mathbb R^n}:=\mathbb R^n\cup\{\infty_{\mathbb R^n}\}$ onto itself such that $y$ and $\infty_{\mathbb R^n}$ are mapped to each other.

To every $\nu\in\mathfrak M^+$ with $\nu(\{y\})=0$ we assign the \emph{Kelvin transform\/}
$\nu^*=\mathcal K_y\nu\in\mathfrak M^+$ (see e.g.\ \cite[Chapter IV, Section 5, n$^\circ$~19]{L}) by means of the formula
\begin{equation}\label{kelv-m}d\nu^*(x^*)=|x-y|^{\alpha-n}\,d\nu(x),\text{ \ where\ }x^*=J_y(x)\in\mathbb R^n.\end{equation}
Then $\mathcal K_y$ is an involution, i.e.
$\mathcal K_y(\mathcal K_y\nu)=\nu$, which implies in view of (\ref{kelv-m}) that
\begin{equation}\label{kelv-mmm}\nu(\mathbb R^n)=\kappa_\alpha\nu^*(y).\end{equation}
Combining (\ref{kelv-m}) and (\ref{inv}) yields
\begin{equation}\label{KP}\kappa_\alpha\nu^*(x^*)=|x-y|^{n-\alpha}\kappa_\alpha\nu(x)\text{ \ for all\ }x^*\in\mathbb R^n,\end{equation}
while multiplying (\ref{kelv-m}) by (\ref{KP}) and then integrating over $\mathbb R^n$ gives
\begin{equation}\label{K}\kappa_\alpha(\nu^*,\nu^*)=\kappa_\alpha(\nu,\nu).\end{equation}

\subsection{Some basic facts on inner Riesz balayage} For $\mu\in\mathfrak M^+$ and $A\subset\mathbb R^n$ arbitrary, $\mu^A\in\mathfrak M^+$ is said to be the (\emph{inner\/}) \emph{balayage\/} of $\mu$ to $A$ \cite[Sections~3, 4]{Z-bal} if
\[\kappa_\alpha\mu^A=\inf\,\kappa_\alpha\nu,\]
the infimum being taken over all $\nu\in\mathfrak M^+$ with the property
\[\kappa_\alpha\nu\geqslant\kappa_\alpha\mu\text{ \ n.e.\ on\ }A.\]
The balayage $\mu^A$ exists and is unique. If moreover $\mu\in\mathcal E_\alpha^+$, then actually
\begin{equation}\label{pr11}\mu^A=P_{\mathcal E_A'}\mu,\end{equation}
$P_{\mathcal E_A'}$ standing for the orthogonal projection in the pre-Hilbert space $\mathcal E_\alpha$ onto $\mathcal E_A'$, the strong closure of $\mathcal E_\alpha^+(A)$:\footnote{Being a strongly closed subset of the strongly complete cone $\mathcal E^+_\alpha$, $\mathcal E_A'$ is strongly complete. Therefore, the orthogonal projection  $P_{\mathcal E_A'}\mu$ exists \cite[Theorem~1.12.3]{E2}.}
\begin{equation*}\label{proj}\|\mu-P_{\mathcal E_A'}\mu\|_\alpha=\min_{\nu\in\mathcal E_A'}\,\|\mu-\nu\|_\alpha,
\end{equation*}
whereas for $\mu\in\mathfrak M^+$ arbitrary, $\mu^A$ can equivalently be determined by the identity
\begin{equation}\label{alternative}\kappa_\alpha(\mu^A,\sigma)=\kappa_\alpha(\mu,\sigma^A)\text{ \ for all\ }\sigma\in\mathcal E^+_\alpha,\end{equation}
$\sigma^A$ being given by (\ref{pr11}).

For the inner balayage $\mu^A$ thus introduced, we actually have
\begin{equation}\label{alternative1}\kappa_\alpha(\mu^A,\theta)=\kappa_\alpha(\mu,\theta^A)\text{ \ for all\ }\theta\in\mathfrak M^+.\end{equation}
Furthermore,\footnote{If moreover $\mu\in\mathcal E^+_\alpha$ and $A$ is closed, then (\ref{ineq1}) characterizes $\mu^A$ uniquely among the measures of the class $\mathcal E^+_\alpha(A)$. However, this no longer holds if either of these two requirements is dropped.}
\begin{align}\label{ineq1}\kappa_\alpha\mu^A&=\kappa_\alpha\mu\text{ \ n.e.\ on\ }A,\\
\label{ineq2}\kappa_\alpha\mu^A&\leqslant\kappa_\alpha\mu\text{ \ on\ }\mathbb R^n.
\end{align}
Also,\footnote{Relations (\ref{v-conv}) and (\ref{pot-conv}) justify the term "inner" balayage.}
\begin{align}&\mu^K\to\mu^A\text{ \ vaguely as\ }K\uparrow A,\label{v-conv}\\
&\kappa_\alpha\mu^K\uparrow\kappa_\alpha\mu^A\text{ \ pointwise on $\mathbb R^n$ as\ }K\uparrow A,\label{pot-conv}\end{align}
where the abbreviation "$K\uparrow A$" means that $K$ increases along the upper directed family $\mathfrak C_A$ of all compact subsets of $A$. The latter implies the monotonicity property:
\begin{equation}\label{mon-pr}\kappa_\alpha\mu^{A_1}\leqslant\kappa_\alpha\mu^{A_2}\text{ \ whenever\ }A_1\subset A_2.\end{equation}

A point $y\in\mathbb R^n$ is said to be (\emph{inner\/}) $\alpha$-\emph{re\-gu\-lar\/} for $A$ if $\varepsilon_y=(\varepsilon_y)^A=:\varepsilon_y^A$, $\varepsilon_y$ being the unit Dirac measure at $y$; the set of all these $y$ is denoted by $A^r$. Then $A^r\subset\overline{A}$,\footnote{We denote by $\overline{A}=\text{\rm Cl}_{\mathbb R^n}A$ and $\partial A$ the Euclidean closure and boundary of a set $A\subset\mathbb R^n$.} since obviously $\varepsilon_x^A\in\mathcal E^+_\alpha$ for all $x\not\in\overline{A}$.
The other points of $\overline{A}$, i.e.
\[y\in\overline{A}\setminus A^r=:A^i,\]
are said to be (\emph{inner\/}) $\alpha$-\emph{ir\-reg\-ul\-ar} for $A$. As seen from (\ref{alternative1}) with $\theta:=\varepsilon_y$,
\begin{equation}\label{reg-pot}y\in A^r\iff\kappa_\alpha\mu^A(y)=\kappa_\alpha\mu(y)\text{ \ for all\ }\mu\in\mathfrak M^+.\end{equation}

By the Winer type criterion \cite[Theorem~6.4]{Z-bal}, $y\in A^{rc}:=(A^r)^c$ if and only if
\begin{equation}\label{w}\sum_{k\in\mathbb N}\,\frac{c_\alpha(A_k)}{q^{k(n-\alpha)}}<\infty,\end{equation}
where $q\in(0,1)$ and $A_k:=A\cap\{x\in\mathbb R^n:\ q^{k+1}<|x-y|\leqslant q^k\}$, while by the Kel\-logg--Ev\-ans type theorem \cite[Theorem~6.6]{Z-bal},\footnote{Observe that both  (\ref{w}) and (\ref{KE}) refer to inner capacity; compare with the Kell\-ogg--Ev\-ans and Wiener type theorems established for \emph{outer\/} balayage (see e.g.\ \cite{Ca2,Br,Doob,BH}). Regarding (\ref{KE}), also note that the set $A^i$ may be of nonzero capacity \cite[Chapter~V, Section~4, n$^\circ$~12]{L}.}
\begin{equation}\label{KE}c_\alpha(A\cap A^i)=c_\alpha(A\setminus A^r)=0.\end{equation}

\subsection{An extended concept of inner Riesz equilibrium measure}\label{sec-ext} For $A\subset\mathbb R^n$ arbitrary, $\gamma_A\in\mathfrak M^+$ is said to be the (\emph{inner\/}) \emph{equilibrium measure\/} \cite[Section~5]{Z-bal} if ($\kappa_\alpha\gamma_A\not\equiv\infty$ and)
\[\kappa_\alpha\gamma_A=\inf\,\kappa_\alpha\nu,\]
the infimum being taken over all $\nu\in\mathfrak M^+$ with $\kappa_\alpha\nu\geqslant1$ n.e.\ on $A$. This $\gamma_A$ is certainly unique (if exists), and according to  \cite[Lemma~5.3]{Z-bal}, it can equivalently be introduced by either of the limit relations
\begin{align}&\gamma_K\to\gamma_A\text{ \ vaguely as\ }K\uparrow A,\notag\\
&\kappa_\alpha\gamma_K\uparrow\kappa_\alpha\gamma_A\text{ \ pointwise on $\mathbb R^n$ as\ }K\uparrow A,\label{v-convv-eq2}\end{align}
the (classical) equilibrium measure $\gamma_K\in\mathcal E^+_\alpha$ on $K\subset A$ compact being defined by means of both (\ref{eq11}) and (\ref{eq22}). Thus
$\kappa_\alpha\gamma_A\leqslant1$ on $\mathbb R^n$, hence $\gamma_A$ is $c_\alpha$-abs\-ol\-utely continuous, though its energy might be infinite.\footnote{In fact, either of $\kappa_\alpha(\gamma_A,\gamma_A)$ and $\gamma_A(\mathbb R^n)$ is finite if and only if $c_\alpha(A)$ is so. For more details, see Section~\ref{sec-intim} below; compare with the classical concept of inner equilibrium measure (Section~\ref{sec-pr1}).}
Furthermore \cite[Lemma~6.11]{Z-bal},
\begin{equation}\kappa_\alpha\gamma_A=1\text{ \ on\ }A^r,\label{equi1}\end{equation}
which combined with (\ref{KE}) gives
\begin{equation}\label{eq-ex0}\kappa_\alpha\gamma_A=1\text{ \ n.e.\ on\ }A.\end{equation}

Section~\ref{sec-intim} below provides a number of equivalent conditions that are necessary and sufficient for the existence of $\gamma_A$.
The approach applied there is based on a close relationship between the concept of inner equilibrium measure and that of inner balayage, described by means of (\ref{har-eq}) with the Kelvin transformation involved.

\subsection{About the results obtained}\label{sec-about} In the current study we first focus on the (\emph{inner\/}) $\alpha$-\emph{har\-mon\-ic measure\/} $\varepsilon_y^A$  of order $\alpha\in(0,2]$ for $A$ arbitrary, which is a natural generalization of the classical concept of ($2$-)har\-mo\-nic measure (see e.g.\ \cite{L,BH,KB0,KB}).

We are motivated by the known fact that the $\alpha$-har\-mo\-nic measure is the main tool in solving the generalized Dirichlet problem for $\alpha$-har\-mon\-ic functions. Indeed, if $A$ is closed while $\partial A$ compact, then for any $f\in C(\partial A)$, the function
\[h_f(y):=\int f\,d\varepsilon_y^A\]
is $\alpha$-har\-mon\-ic on $A^c$ \cite[Chapter~IV, Section~5, n$^\circ$~21]{L} and has the property
\[\lim_{y\to z,\ y\in A^c}\,h_f(y)=f(z)\text{ \ for all\ }z\in A^r,\]
cf.\ \cite[Proposition~VI.11.1]{BH}.

We verify the last relation for an arbitrary set $A\subset\mathbb R^n$ and an arbitrary test function $f\in C_0$, thereby establishing the vague continuity\footnote{For the terminology used here we refer the reader to Bourbaki \cite[Chapter~V, Section~3, n$^\circ$~1]{B2}.} of the map $y\mapsto\varepsilon_y^A$ outside the inner $\alpha$-irregular points for $A$ (Theorem~\ref{pr1}). Furthermore, we describe the Euclidean support of the inner $\alpha$-har\-mon\-ic measure $\varepsilon_y^A$ (Theorem~\ref{desc-sup}), provide a formula for evaluation of its total mass (Theorem~\ref{har-tot}), and obtain necessary and sufficient conditions for $\varepsilon_y^A$ to be of finite energy (more generally, for $\varepsilon_y^A$ to be $c_\alpha$-abs\-ol\-ut\-ely continuous) as well as for $\varepsilon_y^A(\mathbb R^n)\equiv1$ to hold (Corollaries~\ref{conc2}, \ref{seven}, and \ref{cor-en-f}; for illustration, see Example~\ref{ex}). Those criteria are given in terms of the newly defined concepts of inner $\alpha$-thin\-ness and inner $\alpha$-ul\-tra\-thin\-ness of $A$ at infinity (see Definitions~\ref{def-th} and \ref{def-th2}), which for $\alpha=2$ and $A$ Borel coincide with the concepts of outer ($2$-)thin\-ness at infinity introduced by Doob \cite[pp.~175--176]{Doob} and Brelot \cite[p.~313]{Brelot}, respectively (see Remark~\ref{rem-comp} for more details and relevant references).

In Section~\ref{sec-int} (see Corollaries~\ref{C}--\ref{c-desc}) we extend some of these results to $\mu^A$ general by means of establishing the integral representation formula
\[\mu^A=\int\varepsilon_y^A\,d\mu(y)\]
as well as the Borel measurability of the set $A^r$ (Theorems~\ref{th-int-rep} and \ref{measur}). The proofs of these two theorems are based, in turn, on the following observation (Proposition~\ref{cor-count}): there is a countable set $S\subset C_0^\infty:=C_0^\infty(\mathbb R^n)$ such that a net $(\mu_k)\subset\mathfrak M^+$ converges vaguely to $\mu_0$ if (and only if)
\[\lim_{k}\,\mu_k(f)=\mu_0(f)\text{ \ for all\ }f\in S.\]

Basically, the same observation enables us to prove that for $A$ arbitrary, there exists a $K_\sigma$-set $A_0\subset A$ such that (see Theorem~\ref{th-id-bor})
\[\mu^A=\mu^{A_0}\text{ \ for all\ }\mu\in\mathfrak M^+,\]
hence
\[A^r=(A_0)^r.\]
Compare with \cite[Proposition~VI.2.2]{BH} on the existence of a $G_\delta$-set $\hat{A}\supset A$ such that $\bar{\mu}^{\hat{A}}=\bar{\mu}^A$, where $\bar{\mu}^A$ denotes the \emph{outer\/} Riesz balayage of $\mu\in\mathfrak M^+$ to $A\subset\mathbb R^n$ investigated by Bliedtner and Hansen \cite{BH} in the general framework of balayage spaces.

We give various applications of Theorem~\ref{th-id-bor}, in particular we establish the vague and strong
continuity of the inner balayage under the exhaustion of $A$ arbitrary by $A_k:=A\cap U_k$, where $(U_k)$ is an increasing sequence of universally measurable sets with the union $\mathbb R^n$ (Theorem~\ref{pr-cont}).
Assuming additionally that the inner Riesz equilibrium measure $\gamma_A$ exists, we conclude by use of the Kelvin transformation
that there is a $K_\sigma$-set $A'\subset A$ having the properties (see Theorem~\ref{cor-eq-reg})
\[\gamma_A=\gamma_{A'}\text{ \ and \ }A^r=(A')^r.\]
Furthermore, $(\gamma_{A_k})$ converges to $\gamma_A$ vaguely, and also converges strongly if moreover $c_\alpha(A)<\infty$ (Theorem~\ref{th-cont-eq}); the latter strengthens Fuglede's result \cite[Theorem~4.2]{F1} obtained for $A$ universally measurable (for further details, see Remark~\ref{rem-ext} below).

\begin{remark}\label{rem-ou} The inner $\alpha$-har\-mon\-ic measure $\varepsilon_y^A$ for $A$ arbitrary coincides with the outer $\alpha$-har\-mon\-ic measure $\bar\varepsilon_y^{A_0}$ for $A_0$, the $K_\sigma$-sub\-set of $A$ given by Theorem~\ref{th-id-bor}:
\begin{equation}\label{intr1}\varepsilon_y^A=\varepsilon_y^{A_0}=\bar\varepsilon_y^{A_0}\text{ \ for all\ }y\in\mathbb R^n\end{equation}
(see Corollary~\ref{cor-outer}), and therefore
\begin{equation}\label{intr2}A^r=(A_0)^r=(A_0)_o^r,\end{equation}
$(A_0)_o^r$ being the set of the outer $\alpha$-regular points for $A_0$. This suggests that the results obtained in the frame of the theory of outer balayage may be useful in the investigation of inner balayage, and the other way around. For instance, due to (\ref{intr1}) and (\ref{intr2}), the vague continuity of the map $y\mapsto\varepsilon_y^A$ as $y\to z\in A^r$ (Theorem~\ref{pr1}), resp.\ the Borel measurability of the set $A^r$ (Theorem~\ref{measur}), can alternatively be proved by applying \cite[Proposition~VI.11.1]{BH}, resp.\ \cite[Proposition~VI.4.1]{BH}.\end{remark}

\begin{remark} Except for Theorem~\ref{pr1}, the main results of this paper are new even for $\alpha=2$; hence, they also present a further development of the theory of inner Newtonian capacities and of inner Newtonian balayage, originated in the pioneer works by Cartan \cite{Ca,Ca2}. (Note that Theorem~\ref{pr1} for $\alpha=2$ was established in \cite[p.~269, Theorem~5]{Ca2}. However, the proof given in \cite{Ca2} cannot be adapted to $\alpha<2$, being essentially based on specific features of the theory of Newtonian potentials.)\end{remark}

\section{$\alpha$-harmonic measure and $\alpha$-Riesz equilibrium measure}\label{sec-intim}

The study of the inner $\alpha$-harmonic measure $\varepsilon_y^A$, carried out in this section, is based on the systematic use of a close relationship between $\varepsilon_y^A$ and the inner $\alpha$-Riesz equilibrium measure for $A^*=A_y^*$, the inverse of $A\setminus\{y\}$ with respect to $S(y,1)$:\footnote{For the notation used here, see Section~\ref{sec-pr1}. When speaking of the inner Riesz equilibrium measure, we always understand it in the extended sense described in Section~\ref{sec-ext}.}
\begin{equation*}\label{inve}A^*:=A_y^*:=J_y(A\setminus\{y\}).\end{equation*}

\subsection{On the $c_\alpha$-absolute continuity of $\varepsilon_y^A$}\label{ss1}
The following theorem establishes a number of equivalent criteria for the existence of the inner equilibrium measure.

\begin{theorem}\label{th-th}For\/ $A\subset\mathbb R^n$ arbitrary, the following\/ {\rm(i)--(v)} are equivalent.\footnote{Each of these (i)--(v) is also equivalent to (vi) and (vi$'$) below (see Corollary~\ref{seven} and footnote~\ref{vi}).}
\begin{itemize}
\item[{\rm (i)}]There is the inner\/ $\alpha$-Riesz equilibrium measure\/ $\gamma_A$ for\/ $A$.
\item[{\rm (ii)}]There is\/ $\nu\in\mathfrak M^+$ with
\[\essinf_{x\in A}\,\kappa_\alpha\nu(x)>0,\]
the infimum being taken over\/ $A$ except for a subset of inner capacity zero.
\item[{\rm (iii)}]For some\/ {\rm(}equivalently, every\/{\rm)} $y\in\mathbb R^n$,
\begin{equation}\label{iii}\sum_{k\in\mathbb N}\,\frac{c_\alpha(A_k)}{q^{k(n-\alpha)}}<\infty,\end{equation}
where\/ $q\in(1,\infty)$ and\/ $A_k:=A\cap\{x\in\mathbb R^n:\ q^k\leqslant|x-y|<q^{k+1}\}$.
\item[{\rm (iv)}]For some\/ {\rm(}equivalently, every\/{\rm)} $y\in\mathbb R^n$,
\begin{equation*}\label{equiva}y\in(A_y^*)^{rc}.\end{equation*}
\item[{\rm (v)}]For some\/ {\rm(}equivalently, every\/{\rm)} $y\in\mathbb R^n$, $\varepsilon_y^{A_y^*}$ is\/ $c_\alpha$-absolutely continuous.
\end{itemize}

Furthermore, if these\/ {\rm(i)--(v)} hold, then for every\/ $y\in\mathbb R^n$, the Kelvin transform\/ $(\gamma_A)^*=\mathcal K_y\gamma_A$ of\/ $\gamma_{A}$ is actually the inner\/ $\alpha$-har\-mo\-nic measure\/ $\varepsilon_y^{A_y^*}$, i.e.
\begin{equation}\label{har-eq}\varepsilon_y^{A_y^*}=(\gamma_A)^*.\end{equation}
\end{theorem}

\begin{proof} This theorem can in fact be derived from \cite[Sections~5, 6]{Z-bal}. Indeed, according to \cite[Lemma~5.3]{Z-bal}, (i) holds if and only if there is $\nu\in\mathfrak M^+$ with $\kappa_\alpha\nu\geqslant1$ n.e.\ on $A$, which by homogeneity reasons is equivalent to (ii). Noting now that the series in (\ref{iii}) converges (or does not converge) simultaneously for all $y\in\mathbb R^n$, we conclude from \cite[Theorem~5.5]{Z-bal} that (i) is also equivalent to (iii).

Further, let $A_k$ be as in (iii), $q\in(1,\infty)$ and $y\in\mathbb R^n$ being arbitrarily chosen. Denoting $x^*:=J_y(x)$ for $x\ne y$, we get from (\ref{inv})
\[q^{-2k-2}|x-z|\leqslant|x^*-z^*|\leqslant q^{-2k}|x-z|\text{ \ for all\ }x,z\in A_k;\]
hence, by \cite[Remark to Theorem~2.9]{L},
\begin{equation}\label{est}q^{-(2k+2)(n-\alpha)}c_\alpha(A_k)\leqslant c_\alpha(A_k^*)\leqslant q^{-2k(n-\alpha)}c_\alpha(A_k),\end{equation}
where \[A_k^*:=J_y(A_k)=A_y^*\cap\bigl\{x\in\mathbb R^n:\ q^{-k-1}<|x-y|\leqslant q^{-k}\bigr\}.\]
Therefore (\ref{iii}) holds if and only if
\[\sum_{k\in\mathbb N}\,q^{k(n-\alpha)}c_\alpha(A_k^*)<\infty,\]
which, by the Winer type criterion of inner $\alpha$-ir\-reg\-ul\-ar\-ity (see (\ref{w})), is equivalent to the inclusion $y\in(A_y^*)^{rc}$. Thus (iii)$\iff$(iv).

Moreover, if a given $y$ is not $\alpha$-reg\-ular for $A_y^*$, or equivalently if the equilibrium measure $\gamma_A$ exists, then, by \cite[Lemma~6.8]{Z-bal}, the $\alpha$-har\-mo\-nic measure $\varepsilon_y^{A_y^*}$ is actually the Kelvin transform of $\gamma_A$, i.e.\ (\ref{har-eq}) holds. Hence $\varepsilon_y^{A_y^*}$ along with $\gamma_A$ is $c_\alpha$-ab\-sol\-utely continuous, the inverse of any $E\subset\mathbb R^n$ with $c_\alpha(E)=0$ having again zero inner capacity \cite[p.~261]{L}, and so (iv) indeed implies (v). The converse is obvious, for if (iv) were not true, $\varepsilon_y^{A_y^*}$ would not be $c_\alpha$-ab\-sol\-utely continuous, being equal to $\varepsilon_y$. \end{proof}

This leads us naturally to the following definition (compare with the definition of outer ($2$-)thin\-ness at infinity by Doob \cite[pp.~175--176]{Doob}).

\begin{definition}\label{def-th} $A\subset\mathbb R^n$ is said to be (\emph{inner\/}) \emph{$\alpha$-thin
at infinity\/} if assertions (i)--(v) in Theorem~\ref{th-th} hold true.\footnote{The concept of $\alpha$-thinness at infinity seems to appear first in our earlier work \cite{Z1}, where $A$ was closed. Due to its intimate relation with balayage, this concept plays an important role in the investigation of condenser problems (see e.g.\ recent papers \cite{DFHSZ2,FZ-Pot1,FZ-Pot2,Z-arx,DFHSZ1}).}\end{definition}

\begin{remark}\label{comment}Thus $A$ is \emph{not\/} $\alpha$-thin at infinity if for \emph{some\/} (equivalently, \emph{every\/}) inversion $J_y$, the finite point $y=J_y(\infty_{\mathbb R^n})$ is $\alpha$-\emph{re\-gu\-lar\/} for $A_y^*$, the $J_y$-im\-age of $A\setminus\{y\}$. If this holds for \emph{no\/} inversion $J_y$, the set $A$ is $\alpha$-thin at infinity.\end{remark}

A criterion for $\varepsilon_y^A$ to be $c_\alpha$-ab\-sol\-ut\-ely continuous now reads as follows.

\begin{corollary}\label{conc2}For any\/ $A\subset\mathbb R^n$ and\/ $y\in\mathbb R^n$, $\varepsilon_y^A$ is\/ $c_\alpha$-ab\-sol\-ut\-ely continuous if and only if\/ $A_y^*$ is\/ $\alpha$-thin at infinity, or equivalently\/ $y\in A^{rc}$. In the affirmative case,
\[\varepsilon_y^A=\mathcal K_y\gamma_{A_y^*}.\]
\end{corollary}

\begin{proof}This follows from Theorem~\ref{th-th} with $A$ and $A_y^*$ interchanged.\end{proof}

\subsection{On the total mass of $\varepsilon_y^A$} We now establish a formula for evaluation of the total mass of the inner $\alpha$-harmonic measure $\varepsilon_y^A$ for $A\subset\mathbb R^n$ arbitrary, and give necessary and sufficient conditions for $\varepsilon_y^A(\mathbb R^n)\equiv1$ to hold.\footnote{In general, $\varepsilon_y^A(\mathbb R^n)\leqslant1$ \cite[Corollary~4.9]{Z-bal}.}

\begin{theorem}\label{har-tot} For every\/ $y\in\mathbb R^n$,
\begin{equation}\label{e-har-tot}\varepsilon_y^A(\mathbb R^n)=\left\{
\begin{array}{ll}\kappa_\alpha\gamma_A(y)&\text{if\/ $A$ is\/ $\alpha$-thin at infinity},\\
1&\text{otherwise}.\\ \end{array} \right.
\end{equation}
\end{theorem}

\begin{proof}Assume first that $A$ is $\alpha$-thin at infinity, or equivalently that the equilibrium measure $\gamma_A$ exists. As $\kappa_\alpha\gamma_A=1$ on $A^r$ by (\ref{equi1}), we may verify the former equality in (\ref{e-har-tot}) only for $y\not\in A^{r}$. For any $K\subset A$ compact, then obviously $y\not\in K^{r}$, hence $\varepsilon_y^K$ is $c_\alpha$-ab\-sol\-utely continuous (Corollary~\ref{conc2}). This implies, in turn, that $\kappa_\alpha\gamma_K=1$ $\varepsilon_y^K$-a.e., $\varepsilon_y^K$ being supported by $K$, and consequently
\[\varepsilon_y^K(\mathbb R^n)=\int1\,d\varepsilon_y^K=\int\kappa_\alpha\gamma_K\,d\varepsilon_y^K=\int\kappa_\alpha(\gamma_K)^K\,d\varepsilon_y=\int\kappa_\alpha\gamma_K\,d\varepsilon_y
=\kappa_\alpha\gamma_K(y),\]
the third equality being true by (\ref{alternative}) with $\mu:=\varepsilon_y$ and $\sigma:=\gamma_K$. Therefore,
\begin{align*}\limsup_{K\uparrow A}\,\kappa_\alpha\gamma_K(y)&=\limsup_{K\uparrow A}\,\varepsilon_y^K(\mathbb R^n)\leqslant\varepsilon_y^A(\mathbb R^n)\\
&{}\leqslant\liminf_{K\uparrow A}\,\varepsilon_y^K(\mathbb R^n)=\liminf_{K\uparrow A}\,\kappa_\alpha\gamma_K(y)=\kappa_\alpha\gamma_A(y),\end{align*}
where the former inequality holds because, by \cite[Corollaries~4.2, 4.9]{Z-bal},
\[\varepsilon_y^K(\mathbb R^n)=(\varepsilon_y^A)^K(\mathbb R^n)\leqslant(\varepsilon_y^A)(\mathbb R^n),\]
the latter inequality is obtained from (\ref{v-conv}) with $\mu:=\varepsilon_y$ in view of the vague lower semicontinuity of the map $\mu\mapsto\int1\,d\mu$ on $\mathfrak M^+$, and the last equality is valid by (\ref{v-convv-eq2}). This establishes the first equality in (\ref{e-har-tot}).

Assume now that $A$ is not $\alpha$-thin at infinity. We may certainly verify the claimed equality $\varepsilon_y^A(\mathbb R^n)=1$ only for $y\in A^{rc}$, for otherwise it is obvious by definition. But then, by Theorem~\ref{th-th} with $A$ and $A_y^*$ interchanged, $A_y^*$ has the equilibrium measure $\gamma_{A_y^*}$,
and moreover $\gamma_{A_y^*}=\mathcal K_y\varepsilon_y^A$. Applying (\ref{kelv-mmm}) with $\nu:=\varepsilon_y^A$ therefore gives
\[\varepsilon_y^A(\mathbb R^n)=\kappa_\alpha\gamma_{A_y^*}(y).\]
The set $A$ being not $\alpha$-thin at infinity implies $y\in(A_y^*)^r$ (Remark~\ref{comment}). By (\ref{equi1}),
\[\kappa_\alpha\gamma_{A_y^*}(y)=1,\]
which substituted into the preceding display completes the proof.\end{proof}

\begin{corollary}\label{seven} $A\subset\mathbb R^n$ is\/ $\alpha$-thin at infinity if and only if the following holds:\footnote{In fact, (vi) is equivalent to the following apparently stronger assertion:
\begin{itemize}
\item[{\rm (vi$'$)}]$c_\alpha(\{y\in\mathbb R^n:\ \varepsilon_y^A(\mathbb R^n)<1\})>0$.
\end{itemize}\label{vi}}
\begin{itemize}
\item[{\rm (vi)}]There is\/ $y\in\mathbb R^n$ with\/
$\varepsilon_y^A(\mathbb R^n)<1$.
\end{itemize}
\end{corollary}

\begin{proof}This follows directly from (\ref{e-har-tot}) because $\kappa_\alpha\gamma_A\ne1$ n.e.\ on $\mathbb R^n$. Indeed, if not, $\gamma_A$ would serve simultaneously as the equilibrium measure on the whole of $\mathbb R^n$, which is however impossible by (\ref{iii}) applied to $\mathbb R^n$.\end{proof}

\subsection{When does $\varepsilon_y^A\in\mathcal E^+_\alpha$ hold?}\label{ss2}
Corollary~\ref{cor-en-f} below provides criteria for the inner $\alpha$-har\-mo\-nic measure $\varepsilon_y^A$ to be of finite energy, thereby specifying Corollary~\ref{conc2} on the $c_\alpha$-abs\-ol\-ute continuity of $\varepsilon_y^A$ for $y\in A^{rc}$.

\begin{lemma}\label{cap-f1}For any\/ $A_1,A_2\subset\mathbb R^n$ with the Euclidean distance\/ $d>0$,
\begin{equation}\label{estt}c_\alpha(A_1)+c_\alpha(A_2)\leqslant c_\alpha(A_1\cup A_2)\left[1+\frac{\max\,\{c_\alpha(A_1),c_\alpha(A_2)\}}{d^{n-\alpha}}\right].\end{equation}
\end{lemma}

\begin{proof}Of course, we can suppose $c_\alpha(A_i)<\infty$ ($i=1,2$). For any $K_i\subset A_i$ compact,
\[\kappa_\alpha(\gamma_{K_1}+\gamma_{K_2})\leqslant M\text{ \ on\ }K_1\cup K_2,\]
where $M$ denotes the value in the square brackets in (\ref{estt}).
By the Val\-l\'{e}e-Pou\-ssin definition of capacity for a compact set \cite[p.~139, Remark]{L}, we thus have
\[c_\alpha(K_1\cup K_2)\geqslant M^{-1}(\gamma_{K_1}+\gamma_{K_2})(\mathbb R^n),\]
which results in (\ref{estt}) by letting $K_i\uparrow A_i$ ($i=1,2$) because $\gamma_{K_i}(\mathbb R^n)\uparrow\gamma_{A_i}(\mathbb R^n)$.\end{proof}

\begin{theorem}\label{th-f-en}For\/ $A\subset\mathbb R^n$ arbitrary, the following\/ {\rm(i$_1$)--(iv$_1$)} are equivalent.
\begin{itemize}
\item[{\rm(i$_1$)}] The inner\/ $\alpha$-Riesz capacity of\/ $A$ is finite:
\[c_\alpha(A)<\infty.\]
\item[{\rm(ii$_1$)}] For some\/ {\rm(}equivalently, every\/{\rm)} $y\in\mathbb R^n$,
\begin{equation}\label{e-cap-f}\sum_{k\in\mathbb N}\,c_\alpha(A_k)<\infty,\end{equation}
where\/ $q\in(1,\infty)$ and\/ $A_k:=A\cap\{x\in\mathbb R^n:\ q^k\leqslant|x-y|<q^{k+1}\}$.
\item[{\rm(iii$_1$)}] For some\/ {\rm(}equivalently, every\/{\rm)} $y\in\mathbb R^n$,
\begin{equation}\label{bal-f-e}\varepsilon_y^{A_y^*}\in\mathcal E^+_\alpha.\end{equation}
\item[{\rm(iv$_1$)}] For some\/ {\rm(}equivalently, every\/{\rm)} $y\in\mathbb R^n$,
\begin{equation}\label{w2}\sum_{k\in\mathbb N}\,\frac{c_\alpha(A_y^*\cap R_k)}{q^{2k(n-\alpha)}}<\infty,\end{equation}
where\/ $q\in(0,1)$ and\/ $R_k:=\{x\in\mathbb R^n:\ q^{k+1}<|x-y|\leqslant q^k\}$.
\end{itemize}
\end{theorem}

\begin{proof}That (ii$_1$) implies (i$_1$) follows directly from a strengthened version of countable subadditivity for inner capacity (see \cite[p.~253]{Ca2} or \cite[p.~158, Remark]{F1}):\smallskip

\emph{For\/ $Q\subset\mathbb R^n$ arbitrary and\/ $U_k\subset\mathbb R^n$, $k\in\mathbb N$, universally measurable,}
\begin{equation}\label{sub}c_\alpha\Bigl(\bigcup_{k\in\mathbb N}\,Q\cap U_k\Bigr)\leqslant\sum_{k\in\mathbb N}\,c_\alpha(Q\cap U_k)\smallskip.\end{equation}

Assuming now $c_\alpha(A)<\infty$, we shall establish (\ref{e-cap-f}) by showing that
\begin{equation*}\label{ser}S_1:=\sum_{k\in\mathbb N}\,c_\alpha(A_{2k-1})<\infty,\quad
S_2:=\sum_{k\in\mathbb N}\,c_\alpha(A_{2k})<\infty.\end{equation*}
Since both these series can be handled in the same manner, we shall verify $S_1<\infty$. Similarly as in \cite[Proof of Lemma~5.5]{L}, repeated application of Lemma~\ref{cap-f1} gives
\begin{equation}\label{long}\sum_{k=1}^N\,c_\alpha(A_{2k-1})\leqslant c_\alpha\Bigl(\bigcup_{k=1}^N\,A_{2k-1}\Bigr)\prod_{k=1}^{N-1}M_k\leqslant c_\alpha(A)\prod_{k=1}^{\infty}M_k\text{ \ for all\ }N\in\mathbb N,\end{equation}
where
\[M_k:=1+\frac{c_\alpha(A)}{(q^{2k+1}-q^{2k})^{n-\alpha}}.\]
As
\[\sum_{k=1}^{\infty}\,\frac1{(q^{2k+1}-q^{2k})^{n-\alpha}}\leqslant\sum_{k=1}^{\infty}\,\frac1{q^{2k(n-\alpha)}}<\infty,\]
the right-hand side in (\ref{long}) is ${}<\infty$. Thus (i$_1$) and (ii$_1$) are indeed equivalent.

Assuming again that (ii$_1$) holds, we get by applying (\ref{est})
\[\sum_{k\in\mathbb N}\,q^{2k(n-\alpha)}c_\alpha\bigl(J_y(A_k)\bigr)<\infty,\]
which is, in fact, (\ref{w2}) with $1/q\in(0,1)$. This proves (ii$_1$)$\Rightarrow$(iv$_1$) by noting that the series in (ii$_1$) converges (or does not converge) simultaneously for all $y\in\mathbb R^n$. The reverse implication can likewise be established with the aid of (\ref{est}).

Let now for a given $y$,  (iv$_1$) hold; then so does (i$_1$). From (\ref{w2}) we get $y\in(A^*_y)^{rc}$, by the Wiener type criterion. Hence $\varepsilon_y^{A_y^*}=\mathcal K_y\gamma_A$ by (\ref{har-eq}), and applying (\ref{K}) yields
\begin{equation}\label{pr-d}\kappa_\alpha(\varepsilon_y^{A_y^*},\varepsilon_y^{A_y^*})=\kappa_\alpha(\gamma_A,\gamma_A)=c_\alpha(A)<\infty,\end{equation}
the inequality being obtained from (i$_1$). Thus (iv$_1$) indeed implies (iii$_1$).

Finally, if (iii$_1$) is true, then certainly $y\in(A^*_y)^{rc}$. As in the preceding paragraph, this again results in (\ref{pr-d}), the inequality being valid by (\ref{bal-f-e}). Thus (iii$_1$)$\Rightarrow$(i$_1$).\end{proof}

This motivates us to introduce the following concept of $\alpha$-ul\-tra\-thin\-ness at infinity.

\begin{definition}\label{def-th2} $A\subset\mathbb R^n$ is said to be (\emph{inner\/}) \emph{$\alpha$-ul\-tra\-thin
at infinity\/} if assertions (i$_1$)--(iv$_1$) in Theorem~\ref{th-f-en} hold true.\end{definition}

Thus $A$ is $\alpha$-ultrathin at infinity if the measure $\varepsilon_y^{A_y^*}$ has finite energy for some (equivalently, every) inversion $J_y$. This holds, in turn, if and only if $c_\alpha(A)<\infty$.

\begin{remark}\label{ultra}If $A$ is $\alpha$-ult\-ra\-thin at infinity, then it is obviously $\alpha$-thin at infinity; but not the other way around, which is clear by comparing Theorem~\ref{th-f-en}(ii$_1$) with Theorem~\ref{th-th}(iii). See also Example~\ref{ex} for illustration and Remark~\ref{rem-comp} for some relevant references.
\end{remark}

A criterion for $\varepsilon_y^A$ to have finite energy now reads as follows.

\begin{corollary}\label{cor-en-f}For\/ $A$ arbitrary, $\varepsilon_y^A$ has finite energy if and only if\/ $A_y^*$ is\/ $\alpha$-ul\-tra\-thin at infinity, or equivalently\/ $c_\alpha(A_y^*)<\infty$. In the affirmative case,
\[\kappa_\alpha(\varepsilon_y^A,\varepsilon_y^A)=\kappa_\alpha(\gamma_{A_y^*},\gamma_{A_y^*})=c_\alpha(A_y^*)<\infty.\]\end{corollary}

\begin{proof}This follows from Theorem~\ref{th-f-en} with $A$ and $A^*_y$ interchanged.\end{proof}

\begin{figure}[htbp]
\begin{center}
\vspace{-.4in}
\hspace{1in}\includegraphics[width=4in]{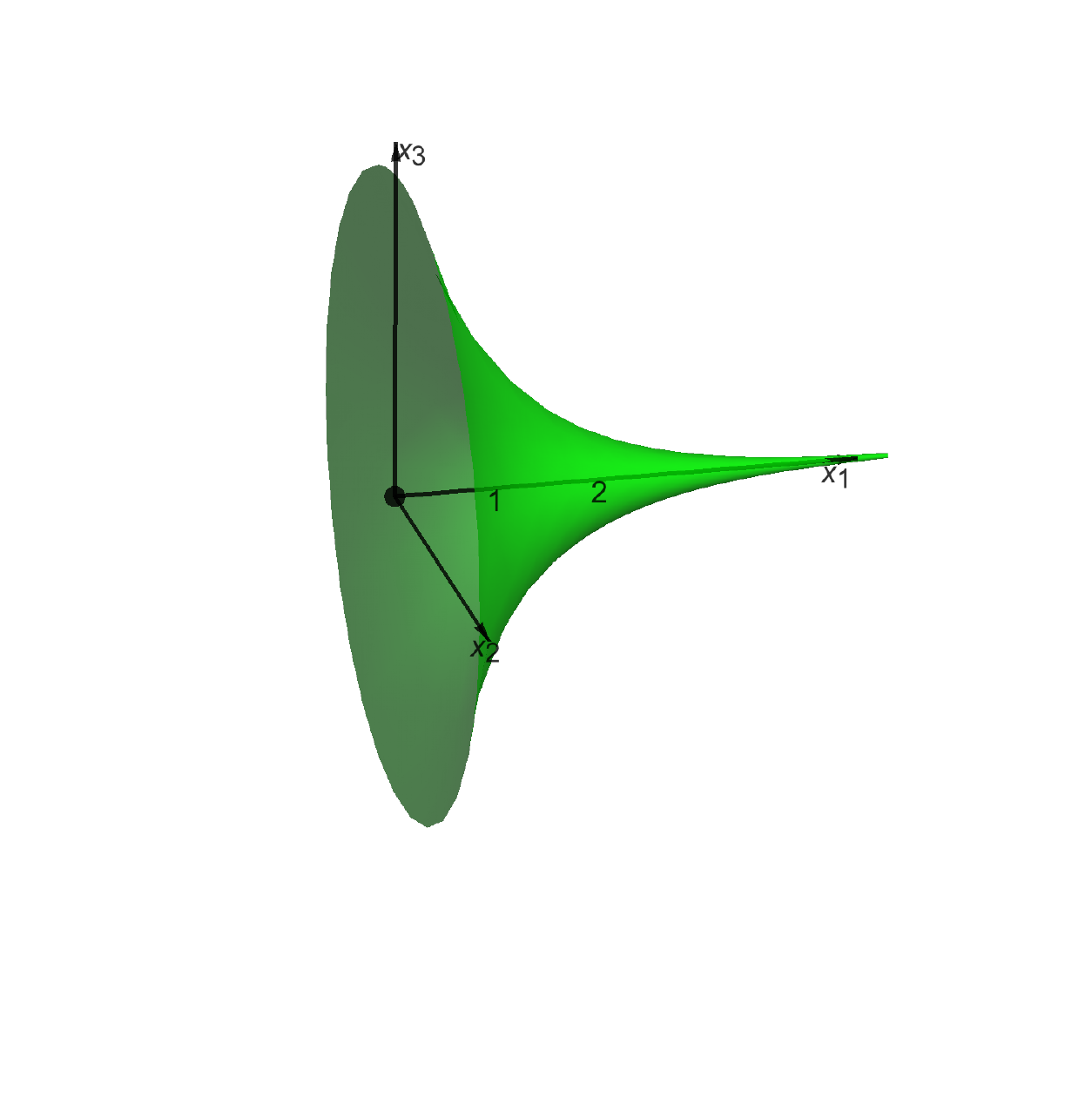}
\vspace{-.8in}
\caption{The set $F_2$ in Example~\ref{ex} with $\varrho_2(x_1)=\exp(-x_1)$.\vspace{-.1in}}
\label{Fig1}
\end{center}
\end{figure}

\begin{example}\label{ex} Let $n=3$ and $\alpha=2$. Consider the rotation bodies
\begin{equation*}\label{descr}F_i:=\bigl\{x\in\mathbb R^3: \ 0\leqslant x_1<\infty, \
x_2^2+x_3^2\leqslant\varrho_i^2(x_1)\bigr\}, \ i=1,2,3,\end{equation*}
where
\begin{align*}
\varrho_1(x_1)&:=x_1^{-s}\text{ \ with\ }s\in[0,\infty),\\
\varrho_2(x_1)&:=\exp(-x_1^s)\text{ \ with\ }s\in(0,1],\\
\varrho_3(x_1)&:=\exp(-x_1^s)\text{ \ with\ }s\in(1,\infty).
\end{align*}
Then $F_1$ is not $2$-thin at infinity, $F_2$ is $2$-thin (but not $2$-ul\-tra\-thin) at infinity (see Figure~\ref{Fig1}),
and $F_3$ is $2$-ul\-tra\-thin at infinity. This follows by combining criteria of $\alpha$-thin\-ness and $\alpha$-ul\-tra\-thin\-ness at infinity, established in Theorems~\ref{th-th} and \ref{th-f-en}, with estimates in \cite[Chapter~V, Section~1, Example]{L}. Hence, by Theorem~\ref{har-tot},
\[\varepsilon_y^{F_1}(\mathbb R^n)=1\text{ \ for all\ }y\in\mathbb R^n,\]
whereas for $i=2,3$, we have the two relations
\begin{align*}\varepsilon_y^{F_i}(\mathbb R^n)&=
\kappa_2\gamma_{F_i}(y)&\hspace{-3.5cm}\text{for all\ }y\in\mathbb R^n,\\
\varepsilon_y^{F_i}(\mathbb R^n)&<1&\hspace{-3.5cm}\text{for all\ }y\in F_i^c,\end{align*}
$\kappa_2\gamma_{F_i}$ being harmonic on the domain $F_i^c$.
Denoting by $F_i^*$ the $J_y$-im\-age of $F_i\setminus\{y\}$ for $y\in\mathbb R^n$ arbitrarily given, we also conclude from Theorems~\ref{th-th} and \ref{th-f-en} that
\begin{equation*}\varepsilon_y^{F^*_i}=\left\{
\begin{array}{ll}\varepsilon_y&\text{for $i=1$},\\
\mathcal K_y\gamma_{F_i}&\text{for $i=2,3$},\\ \end{array} \right.
\end{equation*}
where the energy of $\varepsilon_y^{F^*_2}$ is infinite (though $\varepsilon_y^{F^*_2}$ is $c_2$-ab\-sol\-ut\-ely continuous), whereas
\[\kappa_2(\varepsilon_y^{F^*_3},\varepsilon_y^{F^*_3})=\kappa_2(\gamma_{F_3},\gamma_{F_3})=c_2(F_3)<\infty.\]
\end{example}

\begin{remark}\label{rem-comp}Throughout this remark, $\alpha=2$ and $A$ is Borel measurable. Then $A$ is inner $2$-thin at infinity by Definition~\ref{def-th} if and only if $A$ is outer ($2$-)thin at infinity by Doob \cite[pp.~175--176]{Doob}, the latter concept being defined via the outer ($2$-)ir\-reg\-ul\-ar\-ity of $y\in\mathbb R^n$ for the inverse $A_y^*$. Such equivalence is obvious in view of the fact that for Borel sets, the concepts of inner and outer irregularity coincide.

Applying \cite[Chapter~1.XI, Theorem~5]{Doob} we thus see that $A$ Borel is inner $2$-thin at infinity if and only if there is a positive superharmonic function $v$ on $\mathbb R^n$ with
\[\lim_{x\in A, \ |x|\to\infty}\,v(x)=\infty.\]
Some other properties of such $A$ can be derived e.g.\ from \cite{Deny,Camera}.

On the other hand, for $A$ Borel, our concept of inner $2$-ul\-tra\-thin\-ness at infinity is equivalent to the concept of outer ($2$-)thin\-ness at infinity by Brelot \cite[p.~313]{Brelot} (which is more restrictive than that by Doob). In fact, as noted in \cite[p.~31]{Brelot2} and \cite[Chapter~IX, Section~6]{Br}, a set is thin at infinity in the sense of \cite[p.~313]{Brelot} if and only if its outer capacity is finite. This yields the claimed equivalence, cf.\ Theorem~\ref{th-f-en}(i$_1$).\end{remark}

\section{On the vague continuity of the map $y\mapsto\varepsilon_y^A$}

\begin{theorem}\label{pr1} For\/ $A\subset\mathbb R^n$ arbitrary, the map\/ $y\mapsto\varepsilon_y^A$ is vaguely continuous outside the set\/ $A^i$ of the inner\/ $\alpha$-ir\-reg\-ul\-ar points. That is, for every\/ $f\in C_0$,
\begin{equation}\label{e-cont}\lim_{y\to z}\,\varepsilon_{y}^A(f)=\varepsilon_z^A(f)\text{ \ for all\ }z\not\in A^i,\end{equation}
hence
\begin{equation}\label{e-cont1}\lim_{y\to z}\,\varepsilon_{y}^A(f)=\varepsilon_z(f)\text{ \ for all\ }z\in A^r.\end{equation}
\end{theorem}

Thus $z\in A^r$ serve as accumulation points of the inner $\alpha$-har\-mo\-nic measure $\varepsilon_y^A$:
\[\varepsilon_y^A\to\varepsilon_z\text{ \ vaguely as\ }y\to z\in A^r.\]

\subsection{Auxiliary results} To prove Theorem~\ref{pr1}, we first note that it is enough to verify the vague convergence only on a certain countable set $S$ of positive test functions from $C_0^{\infty}$, while every $f\in C_0^{\infty}$ can be thought of as the potential of a (signed) measure of finite energy. More precisely, the following auxiliary assertions are true.

\begin{lemma}\label{l-count}There is a countable set\/ $S$ of positive functions from\/ $C_0^{\infty}$ having the following property: for every\/ $f\in C^+_0$,\footnote{As usual, $C_0^+$ stands for the class of all positive functions from $C_0$.} there exist a sequence\/ $(f_k)\subset S$ and a function\/ $\varphi\in S$ such that, for every number\/ $\varepsilon>0$,
\[|f_k-f|\leqslant\varepsilon\varphi\text{ \ for all\/ $k$ large enough}.\]\end{lemma}

\begin{proof}This lemma, but with a countable set $S_0\subset C_0^+$ in place of $S\subset C_0^\infty$, is actually given in Bourbaki \cite[Chapter~V, Section~3, n$^\circ$~1, Lemma]{B2}, the space $\mathbb R^n$ being se\-cond-count\-able. What is claimed now, follows by approximating every $g\in S_0$ in the topology on $C_0$ by a sequence $(f_k)$ of positive $f_k\in C_0^{\infty}$ (obtained by regularization \cite[p.~22]{S}) and replacing the set $S_0$ by the (countable) union $S$ of all those $(f_k)$.\end{proof}

\begin{proposition}\label{cor-count}A net\/ $(\mu_j)\subset\mathfrak M^+$ converges to\/ $\mu$ vaguely if\/ {\rm(}and only if\/{\rm)}
\begin{equation}\label{g}\mu_j(g)\to\mu(g)\text{ \ for all\ }g\in S,\end{equation}
where\/ $S$ is the countable set of positive functions from\/ $C_0^{\infty}$ given by Lemma\/~{\rm\ref{l-count}}. Therefore, any two\/ $\mu,\nu\in\mathfrak M^+$ are equal if and only if\/ $\mu(g)=\nu(g)$ for all\/ $g\in S$.\end{proposition}

\begin{proof}Fix $f\in C_0^+$ and $\delta\in(0,\infty)$. We need to show that under assumption (\ref{g}),
\[|\mu_j(f)-\mu(f)|<\delta\text{ \ for all $j$ large enough}.\]
For the given $f$, choose a sequence $(f_k)\subset S$ and a function $\varphi\in S$ as stated in Lemma~\ref{l-count}. Replacing if necessary $\mu$ and $\mu_j$ by shifted measures, we may assume that $\mu(\varphi)>0$. For every $\varepsilon<\delta/6\mu(\varphi)$, then $|f-f_k|\leqslant\varepsilon\varphi$ for some $k\in\mathbb N$, whence
\begin{equation*}\label{g1}|\mu(f)-\mu(f_k)|\leqslant\int|f-f_k|\,d\mu\leqslant\varepsilon\mu(\varphi)<\delta/3.\end{equation*}
Furthermore, for this $k$ and all $j$ large enough,
\begin{align}\label{g3}&|\mu_j(f)-\mu_j(f_k)|\leqslant\int|f-f_k|\,d\mu_j\leqslant\varepsilon\mu_j(\varphi)\leqslant2\varepsilon\mu(\varphi)<\delta/3,\\
\label{g4}&|\mu_j(f_k)-\mu(f_k)|<\delta/3,
\end{align}
where (\ref{g4}) and the third inequality in (\ref{g3}) hold by (\ref{g}) applied to $f_k$ and $\varphi$, respectively. Combining the last three displays gives the estimate required.
\end{proof}

\begin{lemma}\label{1.1}For any\/ $f\in C_0^{\infty}$, there is an absolutely continuous\/ {\rm(}with respect to Lebesgue measure\/{\rm)} signed Radon measure\/ $d\psi=\psi\,dx$ of finite energy such that
\begin{align}\label{e-1}&f=\kappa_\alpha\psi=\kappa_\alpha\ast\psi,\\
\label{psiinf}&\psi(x)=O(|x|^{-n-\alpha})\text{ \ as\ }|x|\to\infty,
\end{align}
where\/ $\ast$ refers to convolution.\footnote{Being locally Lebesgue integrable, $\kappa_\alpha$ is viewed here as density of an absolutely continuous measure.}
\end{lemma}

\begin{proof}For a given $f\in C_0^{\infty}$, write $\psi:=\kappa_{-\alpha}\ast f$, where $\kappa_{-\alpha}(x):=|x|^{-\alpha-n}$ is the distribution defined in \cite[Chapter~I, Section~1, n$^\circ$~2]{L} by means of analytic continuation. Then $\psi\in C^\infty$ by \cite[p.~166]{S}, $f$ having compact support. According to \cite[Lemma~1.1]{L}, both (\ref{e-1}) and (\ref{psiinf}) hold true, whence
\begin{equation}\label{psi}
\psi^{\pm}(x)\leqslant M\min\,\bigl\{1,|x|^{-n-\alpha}\bigr\},
\end{equation}
$M\in(0,\infty)$ being a constant.
It remains to show that $\nu\in\mathcal E_\alpha^+$, where $d\nu:=\psi^\pm\,dx$. (See \cite[pp.~132--133]{FZ} for the proof given below.)

Denote by $\overline{B}=\overline{B}(0,1)$ the closed unit ball in $\mathbb R^n$ and by $\nu_0$ and $\nu_1$ the restrictions of $\nu$ to $\overline{B}$ and $\overline{B}^c$, respectively. Then
$\kappa_\alpha\nu_0=\kappa_\alpha\ast(1_{\overline{B}}\psi^{\pm})$, $1_{\overline{B}}$ being the indicator function for $\overline{B}$, is bounded on $\overline{B}$, hence $\nu_0\in\mathcal E_\alpha^+$.

To prove that $\nu_1\in\mathcal E_\alpha^+$, consider the Kelvin transform $\nu_1^*$ of $\nu_1$ with respect to $S(0,1)$ (noting that $\nu_1(\{0\})=0$), and observe from  (\ref{psi}) by use of (\ref{kelv-m}) that
\[d\nu_1^*(x^*)=|x|^{\alpha-n}1_{\overline{B}\,^c}(x)\psi^{\pm}(x)\,dx\leqslant M|x|^{\alpha-n}|x|^{-n-\alpha}\,dx=M|x|^{-2n}\,dx=M\,dx^*,
\]
the last equality being valid because $|x|^{-n}\,dx=|x^*|^n\,dx^*$. (In fact, write $x=r\xi$, where $r:=|x|$ and $\xi$ ranges over the sphere $S(0,1)$ endowed with its surface measure $d\xi$. Then $dx=r^{n-1}\,dr\,d\xi$ and similarly $dx^*=(r^*)^{n-1}\,dr^*\,d\xi$ with $r^*:=|x^*|=r^{-1}$, hence $dr^*=-r^{-2}\,dr$. We may neglect the minus sign (change of orientation) and conclude that indeed $dx^*=|x|^{-2n}\,dx$.) Thus the situation for $\nu_1^*$ is essentially the same as above for $\nu_0$, both being supported by the ball $\overline{B}$ and having bounded density. Therefore, $\nu_1^*$ has finite energy; hence so does $\nu_1$, by (\ref{K}).
\end{proof}

\subsection{Proof of Theorem~\ref{pr1}} Basing on Proposition~\ref{cor-count} and Lemma~\ref{1.1}, suppose that $f\in C_0^\infty$ and choose $\psi\in\mathcal E_\alpha$ with $\kappa_\alpha\psi=f$. Then
\[\nu:=\psi^{\pm}\in\mathcal E^+_\alpha,\]
which by (\ref{alternative}) with $\mu:=\varepsilon_y$ and $\sigma:=\nu=\psi^\pm$ gives
\begin{equation}\label{meas-inf}\varepsilon_y^A(f)
=\int\kappa_\alpha\psi\,d\varepsilon_y^A
=\int\kappa_\alpha\psi^A\,d\varepsilon_y
=\kappa_\alpha\psi^A(y).
\end{equation}
When varying $y$, $\kappa_\alpha\psi^A(y)$ is a finite-valued continuous function of $y\in\overline{A}^c$ (because $(\psi^{\pm})^A$ is supported by $\overline{A}$), hence so is $\varepsilon_y^A(f)$. Thus (\ref{e-cont}) with $z\in\overline{A}^c$ indeed holds.

Let now $z\in A^r$. Similarly as above, (\ref{e-cont1}) will follow if we show
\begin{equation}\label{chain}\lim_{y\to z}\,\kappa_\alpha\nu^A(y)=\kappa_\alpha\nu^A(z)=\kappa_\alpha\nu(z),\end{equation}
the latter equality being valid by (\ref{reg-pot}).

Assume first that $\kappa_\alpha\nu$ is continuous on $\mathbb R^n$. Since $\kappa_\alpha\nu^A$ is lower semicontinuous (l.s.c.) and not greater than $\kappa_\alpha\nu$ on $\mathbb R^n$, cf.\ (\ref{ineq2}), we have
\begin{align*}\kappa_\alpha\nu^A(z)&\leqslant\liminf_{y\to z}\,\kappa_\alpha\nu^A(y)\leqslant\limsup_{y\to z}\,\kappa_\alpha\nu^A(y)\\{}&\leqslant\limsup_{y\to z}\,\kappa_\alpha\nu(y)=\kappa_\alpha\nu(z)=\kappa_\alpha\nu^A(z),\end{align*}
which establishes (\ref{chain}).

If $\kappa_\alpha\nu$ now is not continuous, choose according to \cite[Theorem~3.7]{L} an increasing sequence $(\mu_k)\subset\mathcal E^+_\alpha$ such that
$\mu_k\to\nu$ vaguely,
$\kappa_\alpha\mu_k\in C(\mathbb R^n)$, and $\kappa_\alpha\mu_k\uparrow\kappa_\alpha\nu$ pointwise on $\mathbb R^n$. Applying (\ref{chain}) to $\mu_k$ in place of $\nu$ then gives
\begin{equation}\label{hryu}\lim_{y\to z}\,\kappa_\alpha\mu_k^A(y)=\kappa_\alpha\mu_k^A(z)=\kappa_\alpha\mu_k(z)\text{ \ for all\ }k\in\mathbb N.\end{equation}
But, as seen from \cite[Proof of Theorem~3.10]{Z-bal} (cf.\ also \cite[Remark~3.11]{Z-bal}), the sequence $(\mu_k^A)$
likewise increases and converges vaguely to $\nu^A$, and moreover
\begin{equation}\label{hryuu}\kappa_\alpha\mu_k^A\uparrow\kappa_\alpha\nu^A\text{ \ pointwise on\ }\mathbb R^n.\end{equation}
Therefore letting $k\to\infty$ in (\ref{hryu}) again results in (\ref{chain}). Indeed,
\begin{align*}\kappa_\alpha\nu^A(z)&\leqslant\liminf_{y\to z, \ k\to\infty}\,\kappa_\alpha\mu_k^A(y)\leqslant\limsup_{y\to z, \ k\to\infty}\,\kappa_\alpha\mu_k^A(y)\\{}&\leqslant\limsup_{k\to\infty}\,\left(\limsup_{y\to z}\,\kappa_\alpha\mu_k^A(y)\right)=\limsup_{k\to\infty}\,\kappa_\alpha\mu_k^A(z)=\kappa_\alpha\nu^A(z),\end{align*}
where the first inequality holds because
$\mu_k^A\otimes\varepsilon_y\to\nu^A\otimes\varepsilon_z$ vaguely as $k\to\infty$ and $y\to z$ (cf.\ \cite[Chapter~III, Section~5, Exercise~5]{B2}), the former equality is valid by (\ref{hryu}), and the latter by (\ref{hryuu}). Thus
\begin{align*}\kappa_\alpha\nu^A(z)=\lim_{y\to z, \ k\to\infty}\,\kappa_\alpha\mu_k^A(y)=\lim_{y\to z}\,\left(\lim_{k\to\infty}\,\kappa_\alpha\mu_k^A(y)\right)=\lim_{y\to z}\,\kappa_\alpha\nu^A(y)\end{align*}
as required.

\section{Support of the inner $\alpha$-harmonic measure}\label{sec-desc} The aim of this section is to describe the Euclidean support $S(\varepsilon_y^A)$ of the inner $\alpha$-har\-mo\-nic measure $\varepsilon_y^A$. While doing this, there is no loss of generality in assuming that $A$ coincides with its \emph{reduced kernel\/} $\breve{A}$, defined as the set of all $x\in A$ such that
$c_\alpha(B(x,r)\cap A\bigr)>0$ for all $r>0$ \cite[p.~164]{L}. In fact, $c_\alpha(A\setminus\breve{A})=0$ by (\ref{sub}), whence
\begin{equation}\label{3}\varepsilon_y^A=\varepsilon_y^{\breve{A}}.\end{equation}
If $\alpha=2$, assume moreover that
\begin{equation}\label{5}c_\alpha(A^i)=0\end{equation} and also, for simplicity, that $\overline{A}^c$ is an (open connected) domain $D$.\footnote{By (\ref{KE}) (the Kell\-ogg--Ev\-ans type theorem), (\ref{5}) is certainly fulfilled if $A$ is closed. For the case where $A$ is closed while $A^c$ disconnected, see \cite[Theorems~7.2, 8.5]{Z-bal}.}

\begin{theorem}\label{desc-sup} Under these hypotheses, if moreover\/ $y\not\in A^r$, then\footnote{If $y\in A^r$, then $S(\varepsilon_y^A)=\{y\}$.}
\begin{equation}\label{lemma-desc-riesz}
S(\varepsilon_y^A)=\left\{
\begin{array}{lll} \overline{A} & \text{if} & \alpha<2,\\ \partial D  & \text{if} & \alpha=2.\\
\end{array} \right.
\end{equation}
\end{theorem}

The proof of Theorem~\ref{desc-sup} (see Section~\ref{sec-proof1}) is based on a description of the support of the inner equilibrium measure $\gamma_A$, given in Theorem~\ref{desc-eq} below.

\begin{theorem}\label{desc-eq}Under the stated  hypotheses, if moreover\/ $A$ is inner\/ $\alpha$-thin at infinity, then
\begin{equation}\label{det2} \kappa_\alpha\gamma_A<1\text{ \ on\ }\overline{A}^c,\end{equation}
\begin{equation}\label{det5}
 S(\gamma_A)=\left\{
\begin{array}{lll} \overline{A}&\text{if}&\alpha<2,\\
\partial D&\text{if}&\alpha=2.\\ \end{array} \right.
\end{equation}
\end{theorem}

\begin{proof} Assuming $\alpha<2$, we shall first prove that
\begin{equation}\label{lesssup}\kappa_\alpha\gamma_A<1\text{ \ on\ }S(\gamma_A)^c.\end{equation}
Suppose this fails for some $x_0\in S(\gamma_A)^c$; then $\kappa_\alpha\gamma_A(x_0)=1$. The potential $\kappa_\alpha\gamma_A$ being $\alpha$-super\-har\-monic on $\mathbb R^n$, $\alpha$-har\-monic on $S(\gamma_A)^c$ \cite[Chapter~I, Section~6, n$^\circ$~20]{L}, and ${}\leqslant1$ on $\mathbb R^n$, applying \cite[Theorem~1.28]{L} gives
$\kappa_\alpha\gamma_A=1$ a.e.\ on $\mathbb R^n$. As follows from the definition of $\alpha$-super\-har\-mon\-ic\-ity, this equality must hold even everywhere on $\mathbb R^n$; so $\gamma_A$ serves as the inner
equilibrium measure on the whole of $\mathbb R^n$, which is however impossible, e.g.\ by Theorem~\ref{th-th}(iii).

To prove the former equality in (\ref{det5}), suppose to the contrary that there is
$x_1\in\overline{A}$ such that $x_1\not\in S(\gamma_A)$, and let $V\subset S(\gamma_A)^c$ be an open neighborhood
of $x_1$. By (\ref{lesssup}), $\kappa_\alpha\gamma_A<1$ on $V$. But $c_\alpha(V\cap A)>0$ in view of $A=\breve{A}$, hence there exists $x_2\in V\cap A$ with  $\kappa_\alpha\gamma_A(x_2)=1$, cf.\ (\ref{eq-ex0}). The contradiction obtained shows that, indeed, $S(\gamma_A)=\overline{A}$.
Substituting this equality into (\ref{lesssup}) establishes (\ref{det2}).

Let now $\alpha=2$. Note from (\ref{equi1}) that under the stated conditions on the set $A$,
\begin{equation}\label{2}\kappa_\alpha\gamma_A=1\text{ \ n.e.\ on\ }\overline{A} \ \bigl({}=D^c\bigr).\end{equation}
If (\ref{det2}) fails,
$\kappa_\alpha\gamma_A$ takes its maximum value $1$ at some $x_3\in D$, hence for all $x\in D$, $\kappa_\alpha\gamma_A$ being harmonic on the domain $D$. Thus (\ref{2}) holds, in fact, n.e.\ on $\mathbb R^n$, which is however impossible, e.g.\ by \cite[Theorem~1.13]{L}. This proves (\ref{det2}).

We also see from (\ref{2}), again by use of \cite[Theorem~1.13]{L}, that $S(\gamma_A)\subset\partial D$.
For the converse, suppose to the contrary that there is $x_4\in\partial D$ such that $x_4\notin S(\gamma_A)$, and let $V_1\subset S(\gamma_A)^c$ be an open neighborhood of $x_4$. As $c_2(V_1\cap\overline{A})>0$, (\ref{2}) implies that $\kappa_2\gamma_A$ takes the value $1$ at some point in $V_1$, hence everywhere on $V_1$, again by the maximum principle. This contradicts (\ref{det2}), $V_1\cap D$ being nonempty.\end{proof}

\subsection{Proof of Theorem~\ref{desc-sup}}\label{sec-proof1}For $y\not\in A^r$, the $J_y$-image $A_y^*$ of $A\setminus\{y\}$ is $\alpha$-thin at infinity, hence  the equilibrium measure $\gamma_{A_y^*}$ exists (Corollary~\ref{conc2} and Theorem~\ref{th-th}). Similarly as it does for $A$, $A_y^*$ coincides with its reduced kernel, the inverse of any $E\subset\mathbb R^n$ with $c_\alpha(E)=0$ being again of zero inner capacity \cite[p.~261]{L}. If $\alpha=2$, then the additional requirements imposed on $A$ do hold also for $A_y^*$.\footnote{We use here the fact that $J_y$ maps $A^i\setminus\{y\}$ onto $(A^*_y)^i$, which can be concluded from the Wiener type criterion of inner $\alpha$-ir\-reg\-ul\-ar\-ity with the aid of (\ref{est}).\label{F}} Applying (\ref{det5}) to $\gamma_{A_y^*}$ we therefore conclude that for $\alpha<2$, $S(\gamma_{A_y^*})=\overline{A_y^*}$, while for $\alpha=2$, $S(\gamma_{A_y^*})$ equals the Euclidean boundary of $J_y(D\setminus\{y\})$. As $\varepsilon_y^A=\mathcal K_y\gamma_{A_y^*}$, this results in (\ref{lemma-desc-riesz}).

\section{Integral representation of inner balayage and applications}\label{sec-int}

\begin{theorem}\label{th-int-rep} For\/ $A\subset\mathbb R^n$ and\/ $\mu\in\mathfrak M^+$ arbitrary,
\begin{equation}\label{L-repr}\mu^A=\int\varepsilon_y^A\,d\mu(y).\end{equation}
\end{theorem}

\begin{proof} Fix $\mu\in\mathfrak M^+$ and note that for any given $f\in C_0^+$, $\varepsilon_y^A(f)$ is a $\mu$-int\-egr\-able function of $y\in\mathbb R^n$. In fact, if $f\in C_0^\infty$, $\varepsilon_y^A(f)$ is Borel measurable, being the difference between two l.s.c.\ functions (cf.\ (\ref{meas-inf})). For $f\in C^+_0$ arbitrary, find by Lemma~\ref{l-count} a sequence $(f_k)\subset C_0^\infty$ and a function $\varphi\in C_0^\infty$, the $f_k$ and $\varphi$ being positive, such that
\[\Bigl|\int(f-f_k)\,d\varepsilon_y^A\Bigr|\leqslant\varepsilon_y^A(\varphi)/k,\]
hence $\varepsilon_y^A(f_k)\to\varepsilon_y^A(f)$ as $k\to\infty$ for all $y\in\mathbb R^n$.\footnote{This convergence is actually even uniform on $\mathbb R^n$, which can be seen by use of $\varepsilon_y^A(\mathbb R^n)\leqslant1$.}
Each of the functions $\varepsilon_y^A(f_k)$, $k\in\mathbb N$, being Borel measurable, so is their pointwise limit $\varepsilon_y^A(f)$ (Egoroff's theorem \cite[Chapter~IV, Section~5, Theorem~2]{B2}).
It thus remains to show that
\[\int\varepsilon_y^A(f)\,d\mu(y)<\infty.\]
Since $\kappa_\alpha\mu\ne\infty$ n.e.\ on $\mathbb R^n$, there is $y_0\in\mathbb R^n$ such that $\kappa_\alpha\mu(y_0)<\infty$ as well as
\[f(x)\leqslant F(x):=M\min\{1,|x-y_0|^{\alpha-n}\}\text{ \ for all\ }x\in\mathbb R^n,\]
$M\in(0,\infty)$ being a constant. According to \cite[Theorem~1.31]{L}, $F=\kappa_\alpha\omega$ on $\mathbb R^n$ for some $\omega\in\mathfrak M^+$, whence
\begin{align*}\int\varepsilon_y^A(f)\,d\mu(y)&=\int d\mu(y)\int f(x)\,d\varepsilon_y^A(x)\leqslant\int d\mu(y)\int\kappa_\alpha\omega(x)\,d\varepsilon_y^A(x)\\
{}&=\int d\mu(y)\int\kappa_\alpha\omega^A(x)\,d\varepsilon_y(x)=\int\kappa_\alpha\omega^A(y)\,d\mu(y)\\
{}&\leqslant\int\kappa_\alpha\omega(y)\,d\mu(y)\leqslant M\int\frac{d\mu(y)}{|y-y_0|^{n-\alpha}}=M\kappa_\alpha\mu(y_0),\end{align*}
the second equality being valid by (\ref{alternative1}) with $\mu:=\omega$ and $\theta:=\varepsilon_y$, and the second inequality by (\ref{ineq2}).
As $\kappa_\alpha\mu(y_0)<\infty$, the $\mu$-int\-egr\-ab\-il\-ity of $\varepsilon_y^A(f)$ follows.

According to \cite[Chapter~V, Section~3, n$^\circ$~1]{B2}, one can therefore define the Radon measure $\nu:=\int\varepsilon_y^A\,d\mu(y)$ on $\mathbb R^n$ by means of the formula
\[
\int f(z)\,d\nu(z)=\int\biggl(\int f(z)\,d\varepsilon_y^A(z)\biggr)\,d\mu(y)\text{ \ for every\ }f\in C_0^+.
\]
Moreover, this identity remains valid when $f$ is allowed to be any positive l.s.c.\ function on $\mathbb R^n$, see \cite[Chapter~V, Section~3, Proposition~2(c)]{B2}. For a given $x\in\mathbb R^n$, we apply this to $f(z)=|x-z|^{\alpha-n}$, $z\in\mathbb R^n$, and thus obtain
\begin{equation}\label{repr-th1}
\kappa_\alpha\nu(x)=\int\biggl(\int|x-z|^{\alpha-n}\,d\varepsilon_y^A(z)\biggr)\,d\mu(y)=\int\kappa_\alpha\varepsilon_y^A(x)\,d\mu(y).
\end{equation}
To establish (\ref{L-repr}), we only need to prove that $\nu=\mu^A$, or equivalently (cf.\ (\ref{alternative}))
\[\kappa_\alpha(\nu,\sigma)=\kappa_\alpha(\mu,\sigma^A)\text{ \ for every\ }\sigma\in\mathcal E^+_\alpha.\]
Using Fubini's theorem we conclude from (\ref{repr-th1}) that indeed
\begin{align*}\kappa_\alpha(\nu,\sigma)&=\int\kappa_\alpha\nu(x)\,d\sigma(x)=\int\biggl(\int\kappa_\alpha\varepsilon_y^A(x)\,d\mu(y)\biggr)\,d\sigma(x)\\
  {}&={\int\biggl(\int\kappa_\alpha\varepsilon_y^A(x)\,d\sigma(x)\biggr)\,d\mu(y)
  =\int\biggl(\int\kappa_\alpha\varepsilon_y(x)\,d\sigma^A(x)\biggr)\,d\mu(y)}\\
  {}&=\int\biggl(\int|x-y|^{\alpha-n}\,d\mu(y)\biggr)\,d\sigma^A(x)
  =\int\kappa_\alpha\mu\,d\sigma^A=\kappa_\alpha(\mu,\sigma^A),
\end{align*}
the fourth equality being valid by (\ref{alternative}) with $\mu:=\varepsilon_y$.\end{proof}

To give some applications of Theorem~\ref{th-int-rep}, we need the following observation.

\begin{theorem}\label{measur}For\/ $A$ arbitrary, the set\/ $A^r$ is Borel measurable.\end{theorem}

\begin{proof}By definition, $A^r$ consists of all $y\in\mathbb R^n$ with $\varepsilon_y^A=\varepsilon_y$, or equivalently with $\varepsilon_y^A(f)=\varepsilon_y(f)$ for all $f\in S$ (Proposition~\ref{cor-count}), where $S\subset C_0^\infty$ is the countable set introduced in Lemma~\ref{l-count}. Denoting by $\mathcal E_\alpha^\circ$ the (countable) set of $\psi\in\mathcal E_\alpha$ such that $f=\kappa_\alpha\psi\in S$ (Lemma~\ref{1.1}), we therefore see that $y$ belongs to $A^r$ if and only if
\[\kappa_\alpha\psi(y)=\int\kappa_\alpha\psi\,d\varepsilon_y=\int\kappa_\alpha\psi\,d\varepsilon_y^A=\int\kappa_\alpha\psi^A\,d\varepsilon_y=\kappa_\alpha\psi^A(y)\text{ \ for all\ }\psi\in\mathcal E_\alpha^\circ,\]
where the third equality holds by (\ref{alternative}) with $\mu:=\varepsilon_y$ and $\sigma:=\psi^\pm$. Being thus a countable intersection of Borel measurable sets, $A^r$ is indeed Borel measurable.\end{proof}

\begin{corollary}\label{b-sum}For any\/ $\mu\in\mathfrak M^+$,
\begin{equation}\label{bal-sum}\mu^A=(\mu|_{A^{rc}})^A+\mu|_{A^r},\end{equation}
the restrictions\/ $\mu|_{A^r}$ and\/ $\mu|_{A^{rc}}$ being well defined in view of Theorem\/~{\rm\ref{measur}}.
\end{corollary}

\begin{proof} Applying (\ref{L-repr}) to $\mu|_{A^r}$ gives
$(\mu|_{A^r})^A=\mu|_{A^r}$, and (\ref{bal-sum}) follows by the additivity of inner balayage.\end{proof}

\begin{corollary}\label{C} $(\mu|_{A^{rc}})^A$ is\/ $c_\alpha$-ab\-sol\-ut\-ely continuous. Hence, $\mu^A$ is\/ $c_\alpha$-ab\-sol\-ut\-ely continuous if and only if\/ $\mu|_{A^r}$ is so.
\end{corollary}

\begin{proof} As seen from (\ref{bal-sum}), it is enough to establish the $c_\alpha$-ab\-sol\-ute continuity of $\mu_0^A$, where $\mu_0:=\mu|_{A^{rc}}$.
Fix a compact set $K\subset\overline{A}$ with $c_\alpha(K)=0$. For every $y\in A^{rc}$, $\varepsilon_y^A$ is $c_\alpha$-ab\-sol\-ut\-ely continuous by Corollary~\ref{conc2}, hence $\varepsilon_y^A(K)=0$. We therefore conclude from (\ref{L-repr}) with $\mu:=\mu_0$ by applying \cite[Chapter~V, Section~3, Theorem~1]{B2} to the indicator function $1_K$ that
\[\int 1_K\,d\mu_0^A=\int\,d\mu_0(y)\int 1_K(x)\,d\varepsilon_y^A(x)=0,\]
which yields the claim.
\end{proof}

\begin{corollary}\label{tm} $A\subset\mathbb R^n$ is not\/ $\alpha$-thin at infinity if and only if
\[\mu^A(\mathbb R^n)=\mu(\mathbb R^n)\text{ \ for all\ }\mu\in\mathfrak M^+.\]
\end{corollary}

\begin{proof}By Corollary~\ref{seven}, it is enough to establish the "only if" part of the corollary. If $A$ is not $\alpha$-thin at infinity, Theorem~\ref{har-tot} gives
\[\varepsilon_y^A(\mathbb R^n)=1\text{ \ for all\ }y\in\mathbb R^n.\]
Applying Theorem~\ref{th-int-rep} we therefore get by \cite[Chapter~V, Section~3, Theorem~1]{B2}
\[\mu^A(\mathbb R^n)=\int 1\,d\mu^A=\int d\mu(y)\int 1(x)\,d\varepsilon_y^A(x)=\int 1\,d\mu=\mu(\mathbb R^n)\text{ \ for all\ }\mu\in\mathfrak M^+,\]
as required.
\end{proof}

\begin{corollary}\label{c-desc}Let\/ $\alpha<2$. For any\/ $A$ and\/ $\mu\in\mathfrak M^+$ with\/ $\mu(A^{rc})>0$,
\[S(\mu^A)=\text{\rm Cl}_{\mathbb R^n}\breve{A}.\]
\end{corollary}

\begin{proof}In view of (\ref{3}), this follows by combining Theorems~\ref{desc-sup} and \ref{th-int-rep}.\end{proof}

\section{For any $A$, there is a $K_\sigma$-set $A_0\subset A$ with $\mu^A=\mu^{A_0}$}\label{sec-last-1}

Inner balayage to $A$ arbitrary can always be reduced to the balayage to $A_0$ Borel. More precisely, the following assertion is true.

\begin{theorem}\label{th-id-bor}For\/ $A$ arbitrary, there exists a\/ $K_\sigma$-set\/ $A_0\subset A$ such that
\begin{equation}\label{eq-a'}\mu^A=\mu^{A_0}\text{ \ for all\ }\mu\in\mathfrak M^+,\end{equation}
hence\/\footnote{In general, $A^i\ne(A_0)^i$.}
\begin{equation}\label{cor-reg}A^r=(A_0)^r.\end{equation}
\end{theorem}

\begin{proof} Fix $\mu\in\mathfrak M^+$. According to (\ref{pot-conv}), the net $(\kappa_\alpha\mu^K)_{K\in\mathfrak C_A}$ ($\mathfrak C_A$ being the upper directed family of all  $K\subset A$ compact) increases to $\kappa_\alpha\mu^A$. The potentials being l.s.c.\ while $\mathbb R^n$ sec\-ond-count\-able, there is an increasing sequence $(K^\mu_k)_{k\in\mathbb N}\subset\mathfrak C_A$ such that
\begin{equation*}\label{lim2}\kappa_\alpha\mu^{K^\mu_k}\uparrow\kappa_\alpha\mu^A\text{ \ pointwise on $\mathbb R^n$ as $k\to\infty$},\end{equation*}
see \cite[Appendix~VIII, Theorem~2]{Doob}. Setting
\begin{equation}\label{prime}A_\mu':=\bigcup_{k\in\mathbb N}K^\mu_k,\end{equation}
we therefore get
\[\kappa_\alpha\mu^{A_\mu'}=\lim_{k\to\infty}\,\kappa_\alpha\mu^{K^\mu_k}=\kappa_\alpha\mu^A\text{ \ on\ }\mathbb R^n,\]
hence
\[\mu^{A_\mu'}=\mu^A.\]
Even more generally, for any $Q$ such that $A_\mu'\subset Q\subset A$, we have
\begin{equation}\label{prime2}\mu^{A_\mu'}=\mu^Q=\mu^A\end{equation}
because, by (\ref{mon-pr}),
\[\kappa_\alpha\mu^{A_\mu'}\leqslant\kappa_\alpha\mu^Q\leqslant\kappa_\alpha\mu^A=\kappa_\alpha\mu^{A_\mu'}\text{ \ on\ }\mathbb R^n.\]

Denoting now by $\mathcal E_\alpha^\circ$ the (countable) set of $\psi\in\mathcal E_\alpha$ with $f=\kappa_\alpha\psi\in S$ (Lemma~\ref{1.1}), where $S\subset C_0^\infty$ is the countable set introduced in Lemma~\ref{l-count}, write
\[A_0:=\bigcup_{\psi\in\mathcal E_\alpha^\circ}A_\psi',\]
$A_\psi'$ being defined by (\ref{prime}) with $\mu:=\psi$. Then $A_0$ is a $K_\sigma$-subset of $A$, and moreover
\[\psi^{A_0}=\psi^A\text{ \ for all\ }\psi\in\mathcal E_\alpha^\circ,\]
by (\ref{prime2}) with $\mu:=\psi$ and $Q:=A_0$. Hence, by (\ref{alternative}) with $\sigma:=\psi$, \[\kappa_\alpha(\mu^A,\psi)=\kappa_\alpha(\mu,\psi^A)=\kappa_\alpha(\mu,\psi^{A_0})=\kappa_\alpha(\mu^{A_0},\psi)\text{ \ for all\ }\psi\in\mathcal E_\alpha^\circ,\]
or equivalently
\[\mu^A(f)=\mu^{A_0}(f)\text{ \ for all\ }f\in S.\]
By Proposition~\ref{cor-count}, this implies (\ref{eq-a'}), the set $A_0$ being independent of $\mu\in\mathfrak M^+$.

Finally, applying (\ref{eq-a'}) to $\mu:=\varepsilon_y$ gives
\[\varepsilon_y=\varepsilon_y^A=\varepsilon_y^{A_0}\text{ \ for all\ }y\in A^r,\]
hence $A^r\subset(A_0)^r$. The opposite being obvious, the proof is complete.\end{proof}

\begin{corollary}\label{cor-outer}For the\/ $K_\sigma$-set\/ $A_0\subset A$ introduced in Theorem\/~{\rm\ref{th-id-bor}},\footnote{See also Remark~\ref{rem-ou} above.}
\[\mu^A=\mu^{A_0}=\bar\mu^{A_0}\text{ \ for all\ }\mu\in\mathfrak M^+,\]
where\/ $\bar\mu^{A_0}$ denotes the outer Riesz balayage of\/ $\mu$ to\/ $A_0$.
\end{corollary}

\begin{proof}In fact, since $A_0$ is the union of an increasing sequence $(K_k)$ of compact sets, from (\ref{eq-a'}) and (\ref{pot-conv}) (applied to $A_0$) we get
\[\kappa_\alpha\mu^A=\kappa_\alpha\mu^{A_0}=\lim_{k\to\infty}\,\kappa_\alpha\mu^{K_k}=\lim_{k\to\infty}\,\kappa_\alpha\bar\mu^{K_k}=\kappa_\alpha\bar\mu^{A_0},\]
the last equality being valid by \cite[Proposition~VI.1.9, Lemma~I.1.7]{BH}.\end{proof}

\begin{theorem}\label{cor-eq-reg}If\/ $A$ is inner\/ $\alpha$-thin at infinity, there is a\/ $K_\sigma$-set\/ $A'\subset A$ with
\begin{align}\label{cor1-1}\gamma_A&=\gamma_{A'},\\
\label{cor1-2}A^r&=(A')^r,\end{align}
where\/ $\gamma_A$, resp.\ $\gamma_{A'}$, is the inner equilibrium measure of\/ $A$, resp.\ $A'$.
\end{theorem}

\begin{proof} For the $J_y$-image $A^*$ of $A\setminus\{y\}$, $y\in\mathbb R^n$ being arbitrarily given, choose a $K_\sigma$-set $A^*_0\subset A^*$ introduced in Theorem~\ref{th-id-bor}, and
write $A':=J_y(A^*_0)$. Combining $\varepsilon_y^{A^*}=\varepsilon_y^{A^*_0}$ with (\ref{har-eq}) yields (\ref{cor1-1}). Since $J_y$ maps $(A^*)^r$ onto $A^r$, and $(A_0^*)^r$ onto $(A')^r$ (see footnote~\ref{F}), we get (\ref{cor1-2}) from (\ref{cor-reg}).\end{proof}

\section{Sequences of inner swept (equilibrium) measures}\label{sec-last}

As further applications of Theorem~\ref{th-id-bor}, we study the vague and strong convergence of sequences of inner swept (resp.\ equilibrium) measures.

Throughout this section, consider the exhaustion of $A$ arbitrary by 
\begin{equation}\label{exh}A_k:=A\cap U_k,\quad k\in\mathbb N,\end{equation}
$(U_k)$ being an increasing sequence of universally measurable sets with the union $\mathbb R^n$.

\begin{theorem}\label{pr-cont} For any\/ $\mu\in\mathfrak M^+$, then\/\footnote{In fact, (\ref{cont2}) can be derived from (\ref{cont1}) by utilizing the monotonicity property (\ref{mon-pr}) and the vague lower semicontinuity of the map $\nu\mapsto\kappa_\alpha\nu$ on $\mathfrak M^+$.}
\begin{align}\label{cont1}&\mu^{A_k}\to\mu^A\text{ \ vaguely},\\
\label{cont2}&\kappa_\alpha\mu^{A_k}\uparrow\kappa_\alpha\mu^A\text{ \ pointwise on\ }\mathbb R^n.
\end{align}
If moreover\/ $\mu\in\mathcal E_\alpha^+$, then also\/ $\mu^{A_k}\to\mu^A$ strongly, i.e.
\begin{equation}\label{cont222}\lim_{k\to\infty}\,\|\mu^{A_k}-\mu^A\|_\alpha=0.\end{equation}
\end{theorem}

\begin{proof} Fix $\mu\in\mathfrak M^+$. By the monotonicity property (\ref{mon-pr}),
\[\kappa_\alpha\mu^{A_k}\leqslant\kappa_\alpha\mu^{A_j}\leqslant\kappa_\alpha\mu^A\text{ \ on $\mathbb R^n$ for $k\leqslant j$},\]
hence there is $\mu_0\in\mathfrak M^+$
such that (\ref{cont1}) and (\ref{cont2}) both hold with $\mu_0$ in place of $\mu^A$ \cite[Theorem~3.9]{L}.
We shall show that $\mu_0=\mu^A$, or equivalently (cf.\ (\ref{alternative}))
\begin{equation}\label{sigma}\kappa_\alpha(\mu_0,\sigma)=\kappa_\alpha(\mu,\sigma^A)\text{ \ for any given\ }\sigma\in\mathcal E^+_\alpha.\end{equation}

Assume first that $A$ is universally measurable; then so are the sets $A_k$. For the given $\sigma\in\mathcal E_\alpha^+$, the balayage $\sigma^{A_k}$ is, in fact, the orthogonal projection of $\sigma$ onto $\mathcal E_{A_k}'$, the strong closure of $\mathcal E^+_\alpha(A_k)$, cf.\ (\ref{pr11}). Thus $\sigma^{A_k}\in\mathcal E_{A_k}'$ and
\[\|\sigma-\sigma^{A_k}\|_\alpha=\min_{\nu\in\mathcal E_{A_k}'}\,\|\sigma-\nu\|_\alpha=\varrho(\sigma,\mathcal E_{A_k}'),\]
where
\[\varrho(\sigma,\mathcal B):=\inf_{\nu\in\mathcal B}\,\|\sigma-\nu\|_\alpha\text{ \ for\ }\mathcal B\subset\mathcal E^+_\alpha.\]
Since $\mathcal E_{A_k}'\subset\mathcal E_{A_j}'\subset\mathcal E_A'$ for $k\leqslant j$, applying \cite[Lemma~4.1.1]{F1} with $\mathcal H:=\mathcal E_\alpha$, $\Gamma:=\{\sigma-\nu:\ \nu\in\mathcal E_{A_j}'\}$, and $\lambda:=\sigma-\sigma^{A_j}$ yields
\[\|\sigma^{A_k}-\sigma^{A_j}\|^2_\alpha=\|(\sigma-\sigma^{A_k})-(\sigma-\sigma^{A_j})\|^2_\alpha\leqslant\|\sigma-\sigma^{A_k}\|_\alpha^2-\|\sigma-\sigma^{A_j}\|_\alpha^2.\]
Being decreasing and lower bounded, the sequence $(\|\sigma-\sigma^{A_k}\|_\alpha)$ is Cauchy in $\mathbb R$, which together with the last display shows that the sequence $(\sigma^{A_k})$ is strong Cauchy in $\mathcal E^+_\alpha$, and hence it converges strongly and vaguely to the unique $\sigma_0\in\mathcal E_A'$, $\mathcal E_A'$ being strongly closed while $\mathcal E^+_\alpha$ strongly complete. This implies that
\begin{align}\label{stream}\varrho(\sigma,\mathcal E_\alpha^+(A))&=\varrho(\sigma,\mathcal E_A')\leqslant\|\sigma-\sigma_0\|_\alpha=\lim_{k\to\infty}\,\|\sigma-\sigma^{A_k}\|_\alpha\\{}&=\lim_{k\to\infty}\,\varrho(\sigma,\mathcal E_{A_k}')=
\lim_{k\to\infty}\,\varrho(\sigma,\mathcal E_\alpha^+(A_k)),\notag\end{align}
the first and last equalities being evident by definition.

The sets $A_k$ being universally measurable, for every $\nu\in\mathcal E_\alpha^+(A)$ we get
\[\lim_{k\to\infty}\,\nu|_{A_k}(f)=\lim_{k\to\infty}\,\int1_{A_k}f\,d\nu=\int1_Af\,d\nu=\nu(f)\text{ \ for all\ }f\in C_0^+,\]
where the second equality holds by \cite[Chapter~IV, Section~1, Theorem~3]{B2}.
Thus $\nu|_{A_k}\to\nu$ vaguely, which gives
\[\|\nu\|_\alpha\leqslant\lim_{k\to\infty}\,\|\nu|_{A_k}\|_\alpha,\quad\kappa_\alpha(\sigma,\nu)\leqslant\lim_{k\to\infty}\,\kappa_\alpha(\sigma,\nu|_{A_k}).\]
The opposite being obvious, equality in fact prevails in these inequalities; hence
\[\|\sigma-\nu\|_\alpha=\lim_{k\to\infty}\,\|\sigma-\nu|_{A_k}\|_\alpha\geqslant\lim_{k\to\infty}\,\varrho(\sigma,\mathcal E_\alpha^+(A_k))\text{ \ for every\ }\nu\in\mathcal E^+_A\]
and consequently
\[\varrho(\sigma,\mathcal E_\alpha^+(A))\geqslant\lim_{k\to\infty}\,\varrho(\sigma,\mathcal E_\alpha^+(A_k)).\]
Combining this with (\ref{stream}) proves that $\sigma_0$, the strong and vague limit of $(\sigma^{A_k})$, is in fact equal to $P_{\mathcal E_A'}\sigma$ $\bigl({}=\sigma^A\bigr)$. This establishes the claimed relations (\ref{cont1})--(\ref{cont222}) for $\mu=\sigma\in\mathcal E^+_\alpha$ and $A$ universally measurable.

Let $A$ now be arbitrary. For the given $\sigma\in\mathcal E^+_\alpha$, choose $K_\sigma$-sets $A'\subset A$ and $A_k'\subset A_k$ so that $A_k'\subset A_{k+1}'$ and
\[\sigma^{A'}=\sigma^A\text{ \ and \ }\sigma^{A_k'}=\sigma^{A_k};\]
that such $A'$ and $A_k'$ exist can be concluded from Theorem~\ref{th-id-bor} and the monotonicity property (\ref{mon-pr}). Writing $\check{A}_k:=A'\cap U_k$, we have $\check{A}_k\subset A_k$, hence
\begin{equation}\label{contt22}\sigma^{A_k'\cup\check{A}_k}=\sigma^{A_k'}=\sigma^{A_k}\text{ \ for all\ }k\in\mathbb N\end{equation}
because, by (\ref{mon-pr}),
\[\kappa_\alpha\sigma^{A_k'}\leqslant\kappa_\alpha\sigma^{A_k'\cup\check{A}_k}\leqslant\kappa_\alpha\sigma^{A_k}=\kappa_\alpha\sigma^{A_k'}.\]
Similarly,
\begin{equation}\label{contt11}\sigma^Q=\sigma^{A'}=\sigma^A,\end{equation}
where
\[Q:=\bigcup_{k\in\mathbb N}\,(A_k'\cup\check{A}_k).\]
The sets $A_k'\cup\check{A}_k$, $k\in\mathbb N$, being universally measurable and forming an increasing sequence with the union $Q$, we conclude from what has been proved above that
\[\sigma^{A_k'\cup\check{A}_k}\to\sigma^{Q}\text{ \ strongly and vaguely},\]
which in view of (\ref{contt22}) and (\ref{contt11}) establishes the theorem for $\mu=\sigma\in\mathcal E^+_\alpha$.

For $\mu\in\mathfrak M^+$ arbitrary, applying (\ref{alternative}) with $\sigma\in\mathcal E^+_\alpha$ gives
\begin{equation}\label{al}\int\kappa_\alpha\mu^{A_k}\,d\sigma=\int\kappa_\alpha\sigma^{A_k}\,d\mu\text{ \ for all\ }k\in\mathbb N.\end{equation}
But, as shown above,
\[\kappa_\alpha\sigma^{A_k}\uparrow\kappa_\alpha\sigma^A\text{ \ and \ }\kappa_\alpha\mu^{A_k}\uparrow\kappa_\alpha\mu_0\text{ \ pointwise on\ }\mathbb R^n.\]
Letting $k\to\infty$ in (\ref{al}) and applying again \cite[Chapter~IV, Section~1, Theorem~3]{B2} we therefore get
\[\int\kappa_\alpha\mu_0\,d\sigma=\int\kappa_\alpha\sigma^A\,d\mu.\]
This establishes (\ref{sigma}), thereby completing the proof of the theorem.\end{proof}

\begin{theorem}\label{th-cont-eq} If\/ $A$ is inner\/ $\alpha$-thin at infinity, then
\begin{align}\label{cont11}&\gamma_{A_k}\to\gamma_A\text{ \ vaguely},\\
\label{cont22}&\kappa_\alpha\gamma_{A_k}\uparrow\kappa_\alpha\gamma_A\text{ \ pointwise on\ }\mathbb R^n,
\end{align}
the sets\/ $A_k$ being given by\/ {\rm(\ref{exh})}. If moreover\/ $A$ is inner\/ $\alpha$-ultrathin at infinity, or equivalently\/ $c_\alpha(A)<\infty$, then\/ $\gamma_{A_k}\to\gamma_A$ also strongly, i.e.
\begin{equation}\label{bor}\lim_{k\to\infty}\,\|\gamma_{A_k}-\gamma_A\|_\alpha=0,\end{equation}
hence
\[\lim_{k\to\infty}\,c_\alpha(A_k)=c_\alpha(A).\]
\end{theorem}

\begin{proof} Fix $y\in\mathbb R^n$. For $E\subset\mathbb R^n$, denote by $E^*$ the $J_y$-image of $E\setminus\{y\}$. Then $(A_k^*)$ is an increasing sequence with the union $A^*$,   hence Theorem~\ref{pr-cont} with $\mu:=\varepsilon_y$ gives
\begin{align}\label{cont1'}&\varepsilon_y^{A_k^*}\to\varepsilon_y^{A^*}\text{ \ vaguely},\\
\label{cont2'}&\kappa_\alpha\varepsilon_y^{A_k^*}\uparrow\kappa_\alpha\varepsilon_y^{A^*}\text{ \ pointwise on\ }\mathbb R^n.
\end{align}
Since, by (\ref{har-eq}),
\[\gamma_{A_k}=\mathcal K_y\varepsilon_y^{A_k^*},\quad\gamma_A=\mathcal K_y\varepsilon_y^{A^*},\]
we derive (\ref{cont22}) from (\ref{cont2'}) by utilizing (\ref{KP}), and (\ref{cont11}) from (\ref{cont1'}) by use of \cite[Lemma~4.3]{L} with $\nu_k:=\varepsilon_y^{A_k^*}$ and $\nu:=\varepsilon_y^{A^*}$. (Observe that this lemma can be applied because $\varepsilon_y^{A_k^*}(\mathbb R^n)\leqslant1$ for all $k$.)

Assume now that $c_\alpha(A)<\infty$. Then all the $\gamma_{A_k}$ and $\gamma_A$ have finite energy, and $\gamma_{A_k}$ minimizes the norm $\|\nu\|_\alpha$ over the class $\Gamma_{A_k}$ of all $\nu\in\mathcal E^+_\alpha$ with $\kappa_\alpha\nu\geqslant1$ n.e.\ on $A_k$ (see Section~\ref{sec-pr1}). Since $\Gamma_A\subset\Gamma_{A_j}\subset\Gamma_{A_k}$ for all $j\geqslant k$, we conclude by applying \cite[Lemma~4.1.1]{F1} with $\mathcal H:=\mathcal E_\alpha$, $\Gamma:=\Gamma_{A_k}$, $\lambda:=\gamma_{A_k}$, and $\mu:=\gamma_{A_j}$ that
\[\|\gamma_{A_j}-\gamma_{A_k}\|^2_\alpha\leqslant\|\gamma_{A_j}\|^2_\alpha-\|\gamma_{A_k}\|^2_\alpha.\]
Being increasing and bounded from above by $c_\alpha(A)<\infty$, the sequence $(\|\gamma_{A_k}\|^2_\alpha)$ is Cauchy in $\mathbb R$, which combined with the preceding display implies that the sequence $(\gamma_{A_k})$ is strong Cauchy in $\mathcal E^+_\alpha$. Thus $(\gamma_{A_k})$ converges strongly (and vaguely) to the unique measure. In view of (\ref{cont11}),  this establishes (\ref{bor}).
\end{proof}

\begin{remark}\label{rem-ext}The latter part of Theorem~\ref{th-cont-eq} generalizes Fuglede's result \cite[Lemma~2.3.3, Theorem~4.2]{F1} established for $A$ universally measurable (cf.\ also \cite[Chapter~II, Section~2, n$^\circ$~9, Remark]{L}). Such generalization can actually be extended further to an arbitrary consistent kernel on a locally compact Hausdorff space that can be represented as a countable union of universally measurable sets. We intend to examine this more closely in forthcoming work.
\end{remark}

\section{Acknowledgements} The author thanks Prof.\ Dr.\ Krzysztof Bogdan and Prof.\ Dr.\ Wolfhard Hansen for many helpful discussions on the topic of the paper.

\end{document}